\documentclass[1p]{elsarticle}

\usepackage[english]{babel}
\usepackage[utf8]{inputenc}
\usepackage[T1]{fontenc}
\usepackage[dvipsnames]{xcolor} 

\usepackage{graphicx} 
\usepackage{tikz, pgfplots, tikz-cd}
\pgfplotsset{compat=1.16}
\usepackage{mathrsfs} 
\usepackage{mathtools, amssymb, amsthm}
\usepackage{bm} 
\usepackage{hyperref, cleveref}
\usepackage{booktabs,siunitx}
\usepackage{algorithm}
\usepackage{algorithmic}
\usepackage{todonotes}
\usepackage{caption}
\usepackage{subcaption}

\title{Learning multivariate functions with low-dimensional structures using polynomial bases}
\author[tuc]{D.~Potts\corref{cor}}
\ead{potts@math.tu-chemnitz.de}
\author[tuc]{M.~Schmischke}
\ead{michael.schmischke@math.tu-chemnitz.de}
\cortext[cor]{Corresponding author}
\address[tuc]{Chemnitz University of Technology, Faculty of Mathematics, 09107 Chemnitz, Germany}

\newtheorem{theorem}{Theorem}
\newtheorem{lemma}[theorem]{Lemma}
\newtheorem{corollary}[theorem]{Corollary}
\newdefinition{remark}{Remark}
\newdefinition{observation}{Observation}
\newproof{prf}{Proof}
\newproof{pot}{Proof of Theorem \ref{thm2}}

\newcommand{\eps}{\mathrm\varepsilon}
\newcommand{\e}{\mathrm e}
\renewcommand{\i}{\mathrm i}
\renewcommand{\b}{\bm}

\DeclareMathOperator*{\argmin}{arg\,min}

\DeclareMathOperator*{\supp}{supp}
\renewcommand{\subset}{\subseteq}

\newcommand{\C}{\mathbb C}
\newcommand{\N}{\mathbb N}
\newcommand{\R}{\mathbb R}
\newcommand{\F}{\b F}

\newcommand{\T}{\mathbb T}
\newcommand{\TI}{[-1,1]}
\renewcommand{\P}{\mathbb P}
\newcommand{\I}{\mathcal{I}}

\newcommand{\X}{\mathcal{X}}

\newcommand{\fun}[3]{#1 \colon #2 \rightarrow #3}
\newcommand{\abs}[1]{\left|#1\right|}
\newcommand{\norm}[2]{\left\| #1 \right\|_{#2}}

\renewcommand{\d}{\,\mathrm{d}}

\newcommand{\bfh}{\hat{\b{f}}}
\renewcommand{\epsilon}{\varepsilon}

\newcommand{\x}{\bm{x}}
\renewcommand{\k}{\bm{k}}

\newcommand{\fc}[1]{\mathrm{c}_{\bm #1}\!\left(f\right)}
\newcommand{\fcc}[2]{\mathrm{c}_{\bm #1}\!\left(#2\right)}

\renewcommand{\cref}{\Cref}
\newcommand{\D}{\mathcal D}
\renewcommand{\u}{\bm u}
\newcommand{\uc}{\bm{u}^{\mathrm c}}
\newcommand{\au}{\abs{\u}}
\newcommand{\auc}{d-\abs{\u}}
\renewcommand{\v}{\bm v}

\newcommand{\av}{\abs{\v}}

\newcommand{\Pud}{\P_{\u}^{(d)}}
\newcommand{\Fud}{\mathbb{F}_{\u}^{(d)}}

\newcommand{\va}[1]{\sigma^2(#1)} 
\newcommand{\gsi}[2]{\varrho(#1,#2)}

\newcommand{\Lom}{\mathrm{L}_2([-1,1],\tilde\omega)}
\newcommand{\Lomu}{\mathrm{L}_2([-1,1]^{\au},\omega)}
\newcommand{\Lomd}{\mathrm{L}_2([-1,1]^d,\omega)}
\newcommand{\Lomdn}{\mathrm{L}_2([-1,1]^8,\omega)}

\renewcommand{\phi}{\varphi}
\renewcommand{\theta}{\vartheta}

\Crefname{observation}{Observation}{Observations}

\begin{document}

\begin{abstract}
In this paper we propose a method for the approximation of high-dimensional functions over finite intervals with respect to complete orthonormal systems of polynomials. An important tool for this is the multivariate classical analysis of variance (ANOVA) decomposition. For functions with a low-dimensional structure, i.e., a low superposition dimension, we are able to achieve a reconstruction from scattered data and simultaneously understand relationships between different variables. 
\end{abstract}


\maketitle

\section{Introduction}\label{sec:int}

The approximation of high-dimensional functions is an active research topic and of high relevance in numerous applications. We assume a setting where we are given scattered data about an unknown function. The related approximation problem is generally referred to as scattered data approximation. Classical methods suffer from the curse of dimensionality in this setting, i.e., the amount of required data increases exponentially with the spatial dimension. Finding ways to circumvent the curse poses the main challenge in this high-dimensional setting. Besides finding an approximation there is the ever more important question of interpretability. In many application one wishes to understand how important the different dimensions and dimension interactions are in order to interpret the results. 

In this paper we consider functions $\fun{f}{[-1,1]^d}{\R}$ defined over the cube with a high spatial dimension $d\in\N$. Given scattered data about $f$, i.e., a finite sampling set $\X \subset [-1,1]^d$ and evaluations $\b y = (f(\x))_{\x\in \X}$, we aim to construct an approximation of $f$ and simultaneously understand its structure, i.e., how important variables and their interactions are. As opposed to black-box approximation or active learning, we may not choose the location of the nodes in $\X$. This prohibits us from using well-established spatial discretizations such as sparse grids, see \cite{Griebel05, Holtz11}, or rank-1 lattices, see \cite{SlJo94, PoVo17, PlPoStTa18}, that use low-dimensional structures in the node set. Our approach to circumvent the curse of dimensionality is to assume sparsity in the (analysis of variance) ANOVA decomposition of the function, i.e., we assume that $f$ is dominated by a small number of low-complexity interactions. This may also be referred to as sparsity-of-effects, see e.g.\ \cite{Wu2011}.

We focus on complete orthonormal systems $\{\phi_{\k}\}$ in $\Lomd$ where the functions are tensor products of univariate polynomials, e.g., the Chebyshev polynomials. Any function from the weighted Lebesgue space $\Lomd$ can then be written as a series $f(\x) = \sum_{\bm k \in \N_0^d} c_{\k} \, \phi_{\k} (\x)$ with coefficients $c_{\k}\in \R$, $\k \in \N_0^d$. Our method focuses on approximations using partial sums of the type $S_I f (\x) = \sum_{\bm k \in I} c_{\k} \, \phi_{\k} (\x)$, with grouped finite index sets $I \subset \N_0^d$ that reflect the low-dimensional structure of $f$. Determining a frequency index set $I$ that yields a good approximation while not scaling exponentially in $d$ poses one of the main challenges. 

The method presented here uses the classical ANOVA decomposition, see \cite{CaMoOw97, RaAl99, LiOw06, Holtz11}, as a main tool. The decomposition is important in the analysis of the dimensions for multivariate, high-dimensional functions. It has also been used in understanding the reason behind the success of certain quadrature methods for high-dimensional integration \cite{Ni92, BuGr04,GrHo10} and also infinite-dimensional integration \cite{BaGn14, GrKuSl16, KuNuPlSlWa17}. The unique and orthogonal ANOVA decomposition decomposes a $d$-variate function in $2^d$ ANOVA terms where each term belongs to a subset of $\{1,2,\dots,d\}$. The terms depends only on the variables in the corresponding subset and the number of these variables is the order of the ANOVA term.

Our method assumes sparsity by restricting the number of possible simultaneous dimension interactions. The knowledge that the function $f$ has a structure such that it can be well approximated using this sparsity assumption is the only information we require a-priori. The approach allows us to learn the basis coefficients by solving a least-squares problem. The problem is hard to solve in general since we are dealing with a large system matrix, but we are able to apply the concept of grouped transformation, see \cite{BaPoSc}, to tackle this issue. In summary, we present a method for the approximation of high-dimensional functions with a low-dimensional structure using possibly noisy scattered data. 

The outline of the paper is as follows. In \cref{sec:prelim} we introduce some necessary preliminaries for weighted Lebesgue spaces with complete orthonormal systems of polynomials. Moreover, we discuss the non-equispaced fast cosine transform and the fast polynomial transform for the evaluation of Chebyshev partial sums and computing the basis exchange from any polynomial bases to the Chebyshev system, respectively. In \Cref{sec:anova} we consider the properties of the ANOVA decomposition in the previously explained setting of weighted Lebesgue spaces. The approximation method itself is discussed in \Cref{sec:approx} with numerical examples in \Cref{sec:numerics}.

\section{Prerequisites, Notation and orthogonal Polynomials}\label{sec:prelim}

Let $\fun{\tilde\omega}{(-1,1)}{\R}$ be a non-negative weight function with $\int_{-1}^1\, \tilde\omega (x)\, \d x \, = \, 1$ then we define the weighted Lebesgue space 
\begin{equation*}
	\Lom \coloneqq \left\{ \fun{f}{[-1,1]}{\R} \colon \norm{f}{\Lom} = \sqrt{\int_{-1}^1  \abs{f(x)}^2 \, \tilde\omega(x) \d x} \right\}
\end{equation*}
with the inner product
$$
  \langle f,g\rangle  := \int_{-1}^1  f(x) g(x) \, \omega(x) \d x.
$$ Moreover, we consider a complete orthonormal system of polynomials $\{ \varphi_{k}  \}_{k \in \N_0}$ in $\Lom$. Here, we have $\varphi_k\in \Pi_k$ with $\Pi_k$ denoting the set of polynomials of degree $\le k$. Taking the products $\varphi_{\k}(\x)\coloneqq\prod_{j=1}^d \varphi_{k_j}(x_j)$ we find that the system $\{\varphi_{\k}\}_{\bm k \in \N_0^d}$ is an orthonormal basis in the tensor product space $\Lomd$ and the functions $f \in \Lomd$ have a unique representation with respect to the system $\{ \varphi_{\k}  \}_{\bm k \in \N^d_0}$ as series
\begin{equation*}
	f(\bm x) = \sum_{\bm k \in \N^d_0} \fc{k} \, \varphi_{\k} (\x),
\end{equation*}
where $\fc{k} \coloneqq \int_{\TI^d} f(\bm x)\, \varphi_{\k} (\x)\,\omega (\bm x)\d\bm x \in \R$, $\bm k \in \N^d_0$, are the basis coefficients of $f$. The density $\omega$ is a product density, i.e., $$\omega(\x) \coloneqq \prod_{s\in\D} \tilde{\omega}(x_s).$$

For a finite index set $\I \subset \N^d_0$, we call
\begin{equation}\label{eq:pre:fps}
	S(\I) f (\x) = \sum_{\k\in \I} \fc{k}\varphi_{\k} (\x),
\end{equation}
the partial sum of $f$ with respect to the index set $\I$. In this paper we make use of the fact, that we are able to compute the sum \eqref{eq:pre:fps} for arbitrary nodes $\x_j \in [-1,1]^d$, $j=1,2,\dots,M$, $M \in \N$, in an efficient manner. We realize this fast evaluation as follows:

Consider the univariate polynomial
$$
P\,:=\, \sum_{k=0}^{N} c_k\, \varphi_k\in \Pi_N
$$
with known real coefficients $c_k$. Our concern is the 
realization of a the basis
exchange from 
$\{ \varphi_k \}_{k=0}^N$ to $\{ T_k \}_{k=0}^N$ in $\Pi_N$
that produces the Chebyshev coefficients $\tilde c_k$ in 
\begin{equation*}
  \label{3.3}
  P\, =\, \sum_{k=0}^{N} \tilde c_k\, T_k\, .
\end{equation*}
By $T_k:=\sqrt{2}^{1-\delta_{k,0}}\cos ( k \arccos \cdot )$, we denote
the normed Chebyshev polynomials of first kind.
Note that $\arccos:\, [-1,1] \to [0,\pi)$ is the inverse
function of $\cos$ restricted to $[0,\pi)$.
As known, the Chebyshev polynomials form a complete orthonormal system in
$\Lom$ with the special
Chebyshev density $\tilde\omega (x):=\pi^{-1} \cdot (1-x^2)^{-1/2}$.
For $m,n \in \N_0$ we have
$$
  \langle T_m ,T_n \rangle  = \left\{
  \begin{array}{ll}
  1 & m=n , \\
  0 & m\neq n\, .
  \end{array}\right.
$$
An algorithm, that realize the fast evaluation of $\tilde c_k$ from $c_k$ is known as discrete polynomial transform  and was developed in \cite{postta98}, see also the approach of
Driscoll and Healy for the transposed problem developed in
\cite{drhe}. Our approach computes the basis exchange with ${\cal O} (N \log^2\! N)$
arithmetical operations by a divide--and--conquer technique combined with
fast polynomial multiplications. The algorithm was designed for arbitrary
polynomials $P_n$ satisfying a three--term recurrence relation, see \cite[Section 6.5]{PlPoStTa18}. We introduce the notation $T_{\k}(\x):=\prod_{j=1}^d T_{k_j}(x_j)$
and observe that this algorithm can be straightforward generalized to the tensor product case,
such that we realize the basis exchange, i.e., compute the Chebyshev coefficients $\tilde c_{\k}\in \R$ from the coefficients $c_{\k}\in \R$, 
\begin{equation*}
	P = \sum_{\k\in \{0,1,\dots,N\}^d} c_{\k}\varphi_{\k} 
	= \sum_{\k\in \{0,1,\dots,N\}^d} \tilde c_{\k} T_{\k} ,
\end{equation*}
in ${\cal O} (N^d \log^{2d}\! N)$ arithmetical operations.
Knowing the Chebyshev coefficients $\tilde c_{\k}$, the values $P(\x_{j})$, $j=0,\ldots,M$, can be computed by the non-equidistant cosine transform at the nodes $\arccos(\x_j)$
by \cite[Algorithm 7.10]{PlPoStTa18} in the complexity of ${\cal O} \bigl(N^d \log  N+  M\bigr)$ 
arithmetical operations.
In summary we are able to compute the polynomial $P$ at all arbitrary nodes $\x_j$, $j=0,\ldots, M$
\begin{equation}\label{eq:pre:P}
	P(\x_j) = \sum_{\k\in \{0,1,\dots,N\}^d} c_{\k}\varphi_{\k} (\x_j),
\end{equation}
in only ${\cal O} (N^d \log^{2d}\! N + M)$ arithmetical operations. For the special case of Chebyshev polynomials, i.e., $\varphi_{\k}=T_{\k}$ we need only ${\cal O} (N^d \log\! N + M)$ arithmetical operations, since the discrete polynomial transform is not necessary.  We stress on the fact, that a fast algorithm 
 implies the factorization of the 
transform matrix 
$\mathbf P:=\left(\varphi_{\k}(\x_j)\right)_{j=0,\ldots,M,\k\in \{0,1,\dots,N\}^d}$ 
into a product of sparse matrices. Consequently,
once a fast algorithm for \eqref{eq:pre:P} is known, a 
fast algorithm for the \lq\lq transposed\rq\rq~problem 
\begin{equation}\label{eq:pre:P2}
	c_{\k} = \sum_{j=0}^M f_j\varphi_{\k} (\x_j),\quad \k\in \{0,1,\dots,N\}^d 
\end{equation}
 with the transform matrix $\mathbf  P^{\mathrm T}$ and the same arithmetical complexity is also available by transposing the sparse matrix product. The algorithms are part of the software package \cite{nfft3}.
 
In order to overcome the high complexity with growing dimensions $d$, we focus on models with a low superposition dimension, see \cref{sec:anova}. To this end we assume, that the effects of degree interactions among the input variables weaken rapidly or vanish altogether.  

\section{Classical Analysis of Variance Decomposition on the Interval}\label{sec:anova}

In this section we introduce the ANOVA decomposition in the setting of weighted Lebesgue spaces with orthonormal polynomials als bases. See also \cite{CaMoOw97, LiOw06, KuSlWaWo09, Holtz11, PoSchm19a}.
For a given spatial dimension $d$ we denote with $\D = \{ 1,2,\dots,d \}$ the set of coordinate indices and subsets as bold small letters, e.g., $\u \subset \D$. The complement of those subsets are always with respect to $\D$, i.e., $\uc = \D \setminus \u$. For a vector $\x \in \C^d$ we define $\x_{\u} = ( x_i )_{i \in \u} \in \C^{\au}$.  Furthermore, we use the $p$-norm (or quasi norm) of a vector which is defined as
\begin{equation*}
\norm{\x}{p} = \begin{cases}
\abs{ \{ i \in \D \colon x_i \neq 0 \} } \quad&\colon p = 0 \\
\left( \sum_{i=1}^d \abs{x_i}^p \right)^{1/p} &\colon 0 < p < \infty \\
\max_{\i\in\D} \abs{x_i} &\colon p = \infty
\end{cases}
\end{equation*}
for $\x\in\R^d$. The space $\Lomd$ with product density $\omega$ and complete orthonormal system $\{\phi_{\k}\}_{\k \in \N_0^d}$ consisting of tensor product functions, see \Cref{sec:int}, is fixed.

We start by defining the integral projection operator
\begin{equation}\label{eq:anova:projection_operator}
\mathrm{P}_{\u} f ( \bm{x}_{\u} ) \coloneqq \int_{\TI^{\auc}}  f(\bm x) \omega(\bm x_{\uc}) \d\bm{x}_{\uc}
\end{equation}
that integrates over the variables $\bm{x}_{\uc}$. Clearly, the image $\mathrm{P}_{\u} f$ depends only on the variables $\bm{x}_{\u} \in \TI^{\au}$. Furthermore, we define the index set
\begin{equation}\label{eq:anova:pud}
\Pud \coloneqq \left\{ \bm k \in \N^d_0 \colon \bm{k}_{\uc} = \bm 0 \right\}
\end{equation}
which can be identified with $\N^{\au}_0$ using the mapping $\bm k \mapsto \bm{k}_{\u}$ as well as the index set
\begin{equation*}
\Fud \coloneqq \left\{ \bm k \in \N_0^d \colon \mathrm{supp}\, \k = \bm u \right\}
\end{equation*}
which can be identified with $\N^{\au}$ using the mapping $\bm k \mapsto \bm{k}_{\u}$. Moreover, we use the convention 
$\N_0^{\abs{\emptyset}} = \{ 0 \}$ and
$\N^{\abs{\emptyset}} = \{ 0 \}$. 
The \textbf{ANOVA term} for $\u \subseteq \D$ is recursively defined as 
\begin{equation}\label{eq:anova:term}
f_{\u} \coloneqq \mathrm{P}_{\u} f - \sum_{\bm v \subsetneq \bm u } f_{\v}.
\end{equation}
We now prove a relationship between the basis coefficients of $\mathrm{P}_{\u} f$, $f_{\u}$ and $f$.
\begin{lemma}\label{lemma:anova:projections}
	Let $f \in \Lomd$ and $\bm\ell \in \N_0^{\au}$. Then
	\begin{equation*}
	\fcc{\ell}{\mathrm{P}_{\u} f} = \fc{k}
	\end{equation*}
	and \begin{equation*}
	\fcc{\ell}{f_{\u} } = \begin{cases}
	\fc{k} \quad&\colon \b\ell \in \N^{\au} \\
	\delta_{\u,\emptyset}\cdot\fc{0} &\colon \b\ell = \b 0 \\
	0 &\colon \text{otherwise}
	\end{cases}
	\end{equation*}
	for $\bm k \in \N_0^d$ with $\bm{k}_{\u} = \bm\ell$ and $\bm{k}_{\uc} = \bm 0$. Moreover, $\mathrm{P}_{\u} f, f_{\u} \in \Lomu$.
\end{lemma}

\begin{proof}
	We prove the formula for $\fcc{\ell}{\mathrm{P}_{\u} f}$, consolidate the two integrals and derive 
	\begin{align*}
	\fcc{\ell}{\mathrm{P}_{\u} f} &= \int_{\TI^{\au}} \int_{\TI^{\auc}}
	  f(\bm x) \omega(\bm x_{\uc}) \d\bm{x}_{\uc}\,
	\varphi_{\bm \ell}(\bm{x}_{\bm u})
	\omega(\bm x_{\u})
	\d\bm{x}_{\u} \\
	&= \int_{\TI^d}  f(\bm x)\,
	\varphi_{\bm \ell}(\bm{x}_{\bm u})
	\omega(\bm x) \d\bm{x} \\
	&= \int_{\TI^d} 	  f(\bm x)\,
	\varphi_{\bm \k}(\bm{x}) \omega(\bm x)
	\d\bm{x} = \fc{k}
	\end{align*}
	for $\bm k \in \N_0^d$ with $\bm{k}_{\u} = \bm\ell$ and $\bm{k}_{\uc} = \bm 0$. Then $\mathrm{P}_{\u} f \in \Lomu$ is clear due to Parseval's identity. 
	
	In order to prove the formula for $\fcc{\ell}{f_{\u}}$, we employ the direct formula for the ANOVA terms $f_{\u}(\x_{\u}) = \sum_{\v \subset \u} (-1)^{\au-\av} \mathrm{P}_{\v} f (\x_{\v})$ to obtain
	\begin{align*}
	\fcc{\ell}{f_{\u}} &= \int_{\T^{\au}} f_{\u}(\x_{\u}) \phi_{\bm\ell}(\x_{\u}) \omega(\bm x_{\u}) \d\x_{\u} \\
	&= \int_{\T^{\au}} \left[\sum_{\v \subset \u} (-1)^{\au-\av} \mathrm{P}_{\v} f (\x_{\v})\right] \phi_{\bm\ell}(\x_{\u}) \omega(\bm x_{\u}) \d\x_{\u} \\
	&= \sum_{\v \subset \u} (-1)^{\au-\av} \int_{\T^{\au}} \mathrm{P}_{\v} f (\x_{\v}) \phi_{\bm\ell}(\x_{\u}) \omega(\bm x_{\u})  \d\x_{\u} \\
	&= \sum_{\v \subset \u} (-1)^{\au-\av} \mathrm{c}_{\k_{\v}}\left(\mathrm{P}_{\v}f\right) \delta_{\k_{\u\setminus\v},\b 0}.
	\end{align*}
	We go on to prove $\fcc{0}{f_{\u} } = \delta_{\u,\emptyset}\cdot\fc{0}$. In this case, $\k_{\v} = \b 0$ and $\delta_{\k_{\u\setminus\v},\b 0} = 1$ for every $\v \subset \u$. By the Binomial Theorem, we have
	\begin{align*}
	\fcc{\ell}{f_{\u}} &= \sum_{\v \subset \u} (-1)^{\au-\av} \mathrm{c}_{\k_{\v}}\left(\mathrm{P}_{\v}f\right) \delta_{\k_{\u\setminus\v},\b 0} = \fc{0} \sum_{\v \subset \u} (-1)^{\au-\av} \\
	&= \fc{0} \sum_{n=0}^{\au} \binom{\au}{n} (-1)^{\au-n} = \fc{0}\cdot\delta_{\u,\emptyset}.
	\end{align*}
	For the second case, we consider an $\b\ell$ and with a set $\overline{\v} \subset \u$ such that $\emptyset\neq\overline{\v} \coloneqq \{ i \in \u \colon k_i = 0 \}\neq\u$. Then $\delta_{\k_{\u\setminus\v},\b 0} = 1 \Longleftrightarrow \overline{\v}^\mathrm{c} \coloneqq \u\setminus\overline{\v} \subset \v$ and with the Binomial Theorem we get
	\begin{align*}
	\fcc{\ell}{f_{\u}} &= \sum_{\v \subset \u} (-1)^{\au-\av} \mathrm{c}_{\k_{\v}}\left(\mathrm{P}_{\v}f\right) \delta_{\k_{\u\setminus\v},\b 0} = \sum_{\overline{\v}^\mathrm{c} \subset \v \subset \u} (-1)^{\au-\av} \mathrm{c}_{\k_{\v}}\left(\mathrm{P}_{\v}f\right) \\
	&= \fc{\k} \sum_{\overline{\v}^\mathrm{c} \subset \v \subset \u} (-1)^{\au-\av} = \fc{\k} \sum_{n=\abs{\overline{\v}^\mathrm{c}}}^{\au} \binom{\au-\abs{\overline{\v}^\mathrm{c}}}{n-\abs{\overline{\v}^\mathrm{c}}} (-1)^{\au-n} \\
	&= \fc{\k} \sum_{m=0}^{\au-\abs{\overline{\v}^\mathrm{c}}} \binom{\au-\abs{\overline{\v}^\mathrm{c}}}{m} (-1)^{\au-\abs{\overline{\v}^\mathrm{c}}-m} = 0.
	\end{align*}
	For the case where the entries of $\b\ell$ are all nonzero, only the addend where $\v = \u$ is nonzero, i.e., $\fcc{\ell}{f_{\u}} = \fc{k}$ and $f_{\u} \in \Lomu$ is clear due to Parseval's identity.
\end{proof} 

Using Lemma \ref{lemma:anova:projections}, we are able to write $\mathrm{P}_{\u} f$ and $f_{\u}$ as both, $d$-dimensional  
$$
\mathrm{P}_{\u} f (\bm x) = \sum_{\bm k \in \Pud} \fc{k} \, \phi_{\k}(\x), \,\,f_{\u}(\bm x) = \sum_{\bm k \in \Fud} \fc{k} \, \phi_{\k}(\x)
$$ 
and $\au$-dimensional series 
$$
\mathrm{P}_{\u} f (\bm{x}_{\u}) = \sum_{\bm \ell \in \N_0^{\au}} \fcc{\ell}{\mathrm{P}_{\u} f} \, \phi_{\bm\ell}(\x_{\u}), 
\,\, f_{\u}  (\bm{x}_{\u}) = \sum_{\bm \ell \in \N^{\au}} \fcc{\ell}{f_{\u} } \, \phi_{\bm\ell}(\x_{\u}).
$$ This directly implies that $\langle f_{\u} , f_{\v} \rangle = 0$ for $\u \neq \v$. With the ANOVA terms we are able to introduce the ANOVA decomposition.
\begin{theorem}\label{Th1}
	Let $f \in \Lomd$, the ANOVA terms $f_{\u}$ as in \eqref{eq:anova:term} and the set of coordinate indices $\D = \{1,2,\dots,d\}$. Then f can be uniquely decomposed as
	\begin{equation}\label{eq:ANOVA:decomp}
	f(\b x) = f_\emptyset + \sum_{i=1}^{d} f_{ \{i\} } (x_i) + \sum_{i = 1}^{d-1} \sum_{j = i+1}^{d} f_{ \{i,j\} } (\b{x}_{\{i,j\}})+ \dots + f_\D(\x) = \sum_{\u \subseteq \D} f_{\b u}(\b{x}_{\u}) 
	\end{equation}
	which we call \textbf{analysis of variance (ANOVA) decomposition}. Moreover, $\bigcup_{\u \subset \D} \Fud = \N_0^d$ and the union is disjoint.
\end{theorem}
\begin{proof}
	We use that $\N_0^d$ is clearly the disjoint union of the sets $\Fud$ for $\u \subset \D$. With this fact we obtain
	\begin{align*}
	\sum_{\u \subset \D} f_{\u}( \x_{\u} ) &= \sum_{\u \subset \D} \sum_{\bm k \in \Fud} \fc{k} \, \varphi_{\k} (\x) = \sum_{\bm k \in \bigcup_{\u \subseteq \D} \Fud} \fc{k} \, \varphi_{\k} (\x) \\
	&= \sum_{\bm k \in \N^d_0} \fc{k} \, \varphi_{\k} (\x)  = f(\b x).
	\end{align*}
	Since the union is disjoint, the decomposition is unique.
\end{proof}

In order to get a notion of the importance of single terms compared to the entire function, we define the \textbf{variance of a function}
\begin{equation*}
\va{f} \coloneqq \int_{\TI^d} \left( f(\x) - \fc{0} \right)^2 \omega(\x) \,\d\x
\end{equation*}
and the equivalent formulation
\begin{equation}\label{eq:gsi:formula}
\va{f} = \norm{f}{\Lomd}^2 - \abs{\fc{0}}^2\, .
\end{equation}
For the ANOVA terms $f_{\u}$ with $\emptyset \neq \u \subset \D$, we have $\fcc{0}{f_{\u}} = 0$ and therefore
$\va{f_{\u}} = \norm{f_{\u}}{\Lomu}^2$.
For $f \in \Lomd$ we obtain the property
	\begin{equation*}
	\va{f} = \sum_{\emptyset\neq\u\subset\D} \va{f_{\u}}
	\end{equation*}
for the variance by Parsevel's identity. In order to measure the importance of a term $f_{\u}$ in relation to the function, we use global sensitivity indices, cf. \cite{So90, So01, LiOw06},
\begin{equation}\label{eq:anova:gsi}
\gsi{\u}{f} \coloneqq \frac{\va{f_{\u}}}{\va{f}} \in [0,1]
\end{equation}
for $\emptyset\neq\u\subset\D$. They have the property $\sum_{\emptyset\neq\u\subset\D} \gsi{\u}{f} = 1$.

The global sensitivity indices motivate the notion of \textbf{effective dimensions} as proposed in \cite{CaMoOw97}. Given a fixed $\delta \in [0,1]$, the \textbf{superposition dimension}, one notion of effective dimension, is defined as 
\begin{equation}\label{eq:anova:superposition}
\min \left\{ s \in \D \colon \sum_{\substack{\emptyset\neq\u \subset \D \\ \au \leq s}} \va{f_{\u}} \geq \delta \va{f}  \right\}
\end{equation} for accuracy $\delta$. In other words, the proportion $\delta$ of the variance $\va{f}$ is explained by ANOVA terms of order less or equal to the superposition dimension.

The number of ANOVA terms in a full decomposition is $\abs{\mathcal{P}(\D)} = 2^d$ and therefore grows exponentially in $d$. This reflects the curse of dimensionality and poses a problem in high-dimensional approximation. In order to circumvent that, we make use of sparsity in the ANOVA decomposition. Specifically, we focus on truncating the ANOVA decomposition, i.e., removing certain terms $f_{\u}$. We therefore define a \textbf{subset of ANOVA terms} as a subset of the power set of $\D$, i.e., $U \subset \mathcal{P}(D)$, such that it is downward closed, i.e., the inclusion condition
\begin{equation}\label{eq:trunc_inclusion}
\u \in U \Longrightarrow \forall \v\subset U \colon \v \in U
\end{equation}
holds, cf.\ \cite[Chapter 3.2]{Holtz11}. This fits with the recursive definition of the ANOVA terms, see \eqref{eq:anova:term}. 
For any subset of ANOVA terms $U$ we then define the \textbf{truncated ANOVA decomposition} as 
\begin{equation*}
\mathrm{T}_{U} f \coloneqq \sum_{\u\in U} f_{\u}.
\end{equation*}
This truncation can be done with the superposition concept in mind, cf.~\eqref{eq:anova:superposition}. For a superposition threshold $d_s \in \D$ we define $U_{d_s} \coloneqq \{ \u \subset \D \colon \abs{\u} \leq d_s \}$ and $\mathrm{T}_{d_s} \coloneqq \mathrm{T}_{U_{d_s}}$. This reduces the number of ANOVA terms to grow polynomially in $d$ for fixed $d_s$ since \begin{equation}\label{polyterms}
\abs{U_{d_s}} \leq \left(\frac{d \cdot \e}{d_s}\right)^{d_s},
\end{equation} cf. \cite{PoSchm19a}. The basis coefficients of the truncated ANOVA decomposition are then
\begin{equation*}
	\fcc{\k}{\mathrm{T}_U f} = \begin{cases}
	\fc{\k} \,\,\,&\colon \exists \u \in U \colon \k \in \Fud \\
	0 &\colon \text{otherwise}.
	\end{cases}
\end{equation*} which means that $\fcc{\k}{\mathrm{T}_{d_s} f}$ is nonzero only for at most $d_s$-sparse frequencies.

The approximation method introduced in \Cref{sec:approx} uses partial sums where the frequency index sets have a grouped structure related to the ANOVA terms in a set $U \subset \mathcal{P}(\D)$. Every finite index set $\I_{\u} \subset \N^d$ corresponds to one ANOVA term $f_{\u}$, $\u \in U$, i.e., \begin{equation*}
	\I_{\u} \subset \left\{ \k \in \N^d \colon \supp \k = \u \right\},
\end{equation*} with $\I_{\emptyset} = \{\b 0\}$ and for the disjoint union we have \begin{equation}\label{freq:union}
	\mathcal{I}(U) \coloneqq \bigcup_{\u \in U} I_{\u} \subset \N_0^d.
\end{equation} It is also possible to choose the frequencies only based on the order of the ANOVA term $\au$, i.e., we have for the projections \begin{equation*}
	\{ \k_{\u} \in \N^{\au} \colon \k \in \I_{\u} \} = \{ \b\ell_{\v} \in \N^{\av}  \colon \b\ell \in \I_{\v} \}
\end{equation*} for every pair of sets $\u, \v \subset \D$ with $\abs{\v} = \abs{\u}$.

\section{Approximation Method}\label{sec:approx}

In this section, we present a method for the approximation of functions $\fun{f}{[-1,1]^d}{\C}$ with a high spatial dimension $d \in \N$ such that $f \in \Lomd$. In scattered data approximation, the data consists of a finite set of sampling nodes $\X = \{ \x_1, \x_2, \dots, \x_M \} \subset [-1,1]^d$ and a vector of values $\b y \in \R^M$. Now, we assume that $y_i \approx f(\x_i)$, i.e., the entries of $\b y$ are noisy evaluations of the function. Here, it is especially important that we cannot choose the location of the nodes $\x_i$. The space $\Lomd$ and the corresponding complete orthonormal system $\{\phi_{\k}\}_{\k\in\N_0^d}$ is fixed. Moreover, we focus on functions with a low-dimensional structure, i.e., a low superposition dimension, cf.~\eqref{eq:anova:superposition}. This implies that choosing a low threshold $d_s \in \D$ will yield a good approximation $ \mathrm{T}_{d_s} f (\b x)\approx f(\b x)$. It has been speculated that functions in many applications consist of a low-dimensional structure and therefore belong to our class. This is referred to as sparsity-of-effects or the Pareto principle, see e.g.~\cite{CaMoOw97, GrKuSl10, Wu2011}. From a theoretical standpoint, functions of specific smoothness classes also have a low-dimensional structure. In \cite{PoSchm19a} it was shown that functions of certain isotropic and dominating-mixed smoothness belong to this class. In particular, a POD (or \textit{product and order-dependent}) weight structure was considered which is motivated by the application of quasi-Monte Carlo methods for PDEs with random coefficients, cf.\ \cite{KuSchwSl12, Graham2014, Kuo2016, Graham2018}.

The idea of the method is to exploit sparsity in the ANOVA decomposition by considering only terms up to order $d_s$, i.e., $\mathrm{T}_{d_s} f$. The immediate benefit is that the number of terms is reduced from being exponential in the spatial dimension $d$ to being polynomial, see \eqref{polyterms}. This assumption also provides us with a way of efficiently calculating an initial least-squares approximation on the basis coefficients. From there we focus on understanding the structure of the function regarding the importance of dimensions and dimension interactions, i.e., the importance of ANOVA terms $f_{\u}$. We measure the importance of a term $f_{\u}$ using the global sensitivity indices $\gsi{\u}{f}$, see~\eqref{eq:anova:gsi}. In order to reduce the number of basis coefficients and subsequently the model complexity further, we use this knowledge to reduce the number of involved ANOVA terms to certain subset $U \subset U_{d_s}$. This simplifies our model function and reduces effects of overfitting.

In \cref{sec:approx:lsqr} we consider how to obtain an approximation on the function given a set of ANOVA terms $U \subset \mathcal{P}(\D)$. This may be the set $U_{d_s}$ for the initial approximation with threshold $d_s \in \D$ or an active set $U \subset U_{d_s}$. The detection of the active set will be addressed in \cref{sec:approx:asd}.

\subsection{Least Squares and Grouped Transformations}\label{sec:approx:lsqr}

In this section we explain the optimization problem for obtaining an approximation on the basis coefficients $\fc{\k}$ of our function $f$ given a set of terms $U \subset \mathcal{P}(\D)$. Here, $U = U_{d_s}$ for the initial approximation and $U \subset U_{d_s}$ after the active set detection. We focus on the Chebyshev system \begin{equation*}
	T_{\k}(\x) \coloneqq \sqrt{2}^{\norm{\k}{0}} \prod_{s \in \supp \k} \cos(k_s \arccos x_s)
\end{equation*} which is a complete orthonormal system in $\Lomd$ with the Chebyshev density \begin{equation}\label{cheb:dens}
	\omega(\x) = \prod_{s \in \D} \frac{1}{\pi\sqrt{1-x_s^2}}
\end{equation} since we may use the FPT to compute the Chebyshev coefficients from other polynomials, see \cref{sec:int}. Now, we approximate $f$ by a finite partial sum $S(\I(U)) f (\x)$, see \eqref{eq:pre:fps}, with corresponding index set $\I(U)$ of a grouped structure, cf.~\eqref{freq:union} for a superposition threshold $d_s \in \D$. The index set for every ANOVA term $f_{\u}$, $\u \in U\setminus\emptyset$, is given by \begin{equation}\label{groupedSets}
	 \I_{\u} = \{  \k \in \N^d \colon \supp \k = \u,  \norm{\k_{\u}}{\infty} \leq N_{\u} - 1 \}
\end{equation} with parameters $N_{\u} \in \N$ and $\I_{\emptyset} = \{\b 0\}$. We then have \begin{equation}\label{eq:system1}
	f(\b x) = \sum_{\bm k \in \N^d_0} \fc{k} \, \phi_{\k} (\x) \approx \mathrm{T}_U f (\x) \approx S(\I(U)) f (\b x) = \sum_{\bm k \in \I(U)} \fc{k} \, \phi_{\k} (\x).
\end{equation} The coefficients $\fc{k}$ are unknown and have to be determined from the given scattered data $\X$ and $\b y$. 

Here, we distinguish between two different cases. The first case being that the nodes $\X$ are distributed i.i.d.~according to the Chebyshev probability density $\omega$ in \eqref{cheb:dens} and the second case being that $\X$ is distributed uniformly in $[-1,1]^d$. 

\subsubsection{Chebyshev Distributed Nodes}

Here, we assume that the nodes $\X$ are distributed i.i.d.~according to the Chebyshev probability density $\omega$. We aim to determine approximations for the basis coefficients by solving the minimization problem \begin{equation}\label{min_problem}
	\bfh_{\text{sol}} = \argmin_{\bfh \in \R^{\abs{\I(U)}}} \norm{\b y - \F(\X,\I(U))\bfh}{2}^2
\end{equation} with system matrix $\F(\X,\I(U)) = (\phi_{\k} (\x))_{\x\in \X,\k \in \I(U)}$. Solving the problem \eqref{min_problem} is equivalent to solving the normal equation \begin{equation}\label{normaleq}
	\F^\top(\X,\I(U)) \F(\X,\I(U)) \bfh_{\text{sol}} = \F^\top(\X,\I(U)) \b y.
\end{equation} The properties of this system have been considered in \cite{KaUlVo19}. To summarize, we get from \cite[Section 5]{KaUlVo19} that the expected value of the matrix product $\F^\top(\X,\I(U)) \F(\X,\I(U))$ is a diagonal matrix and the singular values of $ \F(\X,\I(U))$ are between $\sqrt{\abs{\X}/2}$ and $\sqrt{3\abs{\X}/2}$. This yields an upper bound for the norm of the Moore-Penrose inverse \begin{equation*}
	\norm{\left(\F^\top(\X,\I(U)) \F(\X,\I(U))\right)^{-1} \F^\top(\X,\I(U)) }{2} < \sqrt{\frac{\abs{\X}}{2}}.
\end{equation*} This holds with a probability of $1-\delta$ if \begin{equation*}
	\abs{I(U)} \leq \frac{1}{2^{d_s}\cdot48(\sqrt{2}-\log\delta) }\cdot\frac{\abs{X}}{\log(2\abs{\X})}.
\end{equation*} We also gain that the matrix $\F(\X,\I(U)$ has full rank and our problem a unique solution $\bfh_{\text{sol}}$ with this probability.

Problem \eqref{min_problem} is difficult to solve in general since we have a large matrix and need an efficient matrix-vector multiplication. However, we have a grouped index set $\I(U)$ which allow us to use the Grouped Transformations idea from \cite{BaPoSc}. If we have an order $\u_1, \u_2, \dots, \u_n$, $n = \abs{U}$, on the ANOVA terms, we may write \begin{equation*}
	\F(\X,\I(U))= \left( \b{F}_1 \,\, \b{F}_2 \,\, \cdots \,\, \b{F}_n \right)
\end{equation*} with $\F_i = (\phi_{\k} (\x))_{\x\in \X,\k \in \I_{\u_i}}$. Therefore, a multiplication of $\F(\X,\I(U))$ with the vector $\bfh = ( \bfh_1, \bfh_2, \dots, \bfh_n ) \in \R^{\abs{\I(U)}}$ can be written as $$ \F(\X,\I(U)) \bfh = \sum_{j=1}^n \b{F}_j \bfh_j. $$ An efficient way of performing the matrix-vector multiplication $\b{F}_j \bfh_j$ can then be realized using a NFCT, see \cref{sec:int}. The same holds true for the adjoint problem, i.e., the multiplication of $\F^\top(\X,\I(U))$ with a vector $\b{f} \in \R^{\abs{\X}}$. Here, we have $$ \F^\top(\X,\I(U)) \b{f} = \begin{pmatrix}
	\b{F}^\top_1 \b{f}\\
	\b{F}^\top_2 \b{f} \\
	\vdots \\
	\b{F}^\top_n \b{f} 
\end{pmatrix}. $$ In this case we can use an adjoint NFCT for efficient multiplications $\b{F}^\top_j\b{f}$. We then proceed to solve \eqref{min_problem} using iterative LSQR, see \cite{PaSa82}, in a matrix-free variant, i.e., $\F(\X,\I(U))$ is not explicitly required by providing the fast grouped transformations algorithm for multiplication of vectors with matrices of type $\F(\X,\I(U))$ and $\F^\top(\X,\I(U))$.

\begin{remark}
	The multiplications $\b{F}_j \bfh_j$ and $\b{F}^\top_j \b{f}$, $j = 1,2,\dots,n$, are all independent of each other which allows us to use parallelization for a fast multiplication. If the computer allows for it, it is possible to calculate all $n$ products and $n$ adjoint products simultaneously.
\end{remark}

The elements of the solution vector $\bfh_{\text{sol}} = (\hat{f}_{\k})_{\k \in \I(U)}$ are the unique least-squares approximation to the basis coefficients, i.e., $\hat{f}_{\k} \approx \fc{k}$, with respect to $\X$ and $\b y$. We then have an approximation by the approximate partial sum \begin{equation}\label{approx_fps}
    S(\X, \I(U)) f (\x) \coloneqq \sum_{\bm k \in \I(U)} \hat{f}_{\k} \, \phi_{\k} (\x) \approx S(\I(U))f(\x).
\end{equation}

\subsubsection{Uniformly Distributed Nodes}\label{sec:uninodes}

In this section we assume that the nodes $\X$ are uniformly i.i.d~distributed in $[-1,1]^d$. If we would proceed in the same way as before, the expected value of the matrix $\F^\top(\X,\I(U)) \F(\X,\I(U))$ from the normal equation \eqref{normaleq} would not be the identity. In other words, our system would not be stable. However, this can be fixed by scaling and preconditioning. 

We choose a padding parameter $\theta \in (0,1)$ and scale the nodes $\X$ such that $$\tilde{\X} \coloneqq \left\{  \begin{pmatrix} (1-\theta) \, x_1 \\ \vdots \\ (1-\theta) \, x_d \end{pmatrix} \colon \x = \begin{pmatrix} x_1 \\ \vdots \\ x_d \end{pmatrix} \in \X \right\},$$i.e., we have uniformly distributed nodes in $[-1+\theta,1-\theta]^d$. Now, we choose our preconditioner as the diagonal matrix \begin{equation*}
	\b W = \mathrm{diag}\left( \sqrt{\omega(\x)} \right)_{\x \in \X} 
\end{equation*} such that we have the minimization problem \begin{equation*}
	\bfh_{\text{sol}} = \argmin_{\bfh \in \R^{\abs{\I(U)}}} \norm{\b W \b y - \b W \F(\X,\I(U))\bfh}{2}^2.
\end{equation*} The normal equation \eqref{normaleq} transforms to \begin{equation}\label{min_problem_uni}
	\F^\top(\X,\I(U)) \b{W}^2 \F(\X,\I(U)) \bfh_{\text{sol}} = \F^\top(\X,\I(U)) \b{W} \b y
\end{equation} with $\b{W}^2 = \b W \cdot \b W$. We denote \begin{equation}\label{matrixH}
	\b H \coloneqq \F^\top(\X,\I(U)) \b{W}^2 \F(\X,\I(U)).
\end{equation}In the following we consider the properties of this system and prove that with this preconditioner we are able to achieve a stable system under certain conditions.

\begin{lemma}\label{lem:dis:1}
	Let $\k, \b\ell \in \N_0^d$, $\k \neq \b\ell$, with $\norm{\k}{0}, \norm{\b\ell}{0} \leq d_s \in \D$ and $\vartheta \in (0,1)$. Then \begin{equation*}
		\abs{\int_{[-1,-1+\vartheta]^d} T_{\k}(\x) T_{\b\ell}(\x) \omega(\x) \d\x} \leq 4^{d_s} \left( \frac{\arccos(1-\vartheta)}{\pi} \right)^d
	\end{equation*} and \begin{equation*}
		\abs{\int_{[1-\vartheta,1]^d} T_{\k}(\x) T_{\b\ell}(\x) \omega(\x) \d\x} \leq 4^{d_s} \left( \frac{\arccos(1-\vartheta)}{\pi} \right)^d.
	\end{equation*}
\end{lemma}
\begin{proof}
	We define $C(k,x) = \cos(k\arccos(x))$ and set $M_1(\k,\b\ell) = \{ s \in \D \colon k_s \neq 0, \ell_s \neq 0 \}$, $M_2(\k,\b\ell) = \{ s \in \D \colon \text{either } k_s = 0 \text{ or } \ell_s = 0 \}$, $M_3(\k,\b\ell) = \{s\in\D\colon k_s=\ell_s=0\}$ . The first integral can be separated as follows \begin{align*}
		\int_{[-1,-1+\theta]^d} T_{\k}(\x) T_{\b\ell}(\x) \omega(\x) \d\x = \prod_{s \in M_1 } \underbrace{\frac{2}{\pi} \int_{-1}^{-1+\theta} C(k_s,x) C(\ell_s,x) \frac{1}{\sqrt{1-x^2}} \d x}_{I_1} \\
		\times \prod_{ s \in M_2 } \underbrace{\frac{\sqrt{2}}{\pi} \int_{-1}^{-1+\theta} C(\max\{k_s,\ell_s\},x) \frac{1}{\sqrt{1-x^2}} \d x}_{I_2} \\
		\times \prod_{ s \in M_3 } \underbrace{\frac{1}{\pi} \int_{-1}^{-1+\theta} \frac{1}{\sqrt{1-x^2}} \d x}_{I_3}.
	\end{align*} We have \begin{equation*}
		I_1 = - \left[ \frac{\sin((k_s - \ell_s) \arccos(x))}{\pi (k_s - \ell_s)}  +  \frac{\sin((k_s + \ell_s) \arccos(x))}{\pi (k_s + \ell_s)} \right]_{-1}^{-1+\vartheta}
	\end{equation*} and since $\arccos(-1) = \pi$ this becomes \begin{equation*}
		I_1 = - \frac{\sin((k_s - \ell_s) \arccos(-1+\vartheta))}{\pi (k_s - \ell_s)}  -  \frac{\sin((k_s + \ell_s) \arccos(-1+\vartheta))}{\pi (k_s + \ell_s)}.
	\end{equation*} Writing $\arccos(-1+\vartheta) = \pi - \rho$ yields \begin{equation*}
		\abs{\sin((k_s - \ell_s)  (\pi - \rho))} = \abs{\sin((k_s - \ell_s)\rho)} \leq \abs{(k_s - \ell_s)\rho}.
	\end{equation*} Therefore, we get \begin{equation*}
		\abs{I_1} \leq \frac{\rho}{\pi} + \frac{\rho}{\pi} = \frac{2\rho}{\pi} = \frac{2\arccos(1-\vartheta)}{\pi}.
	\end{equation*} For the second integral we have w.l.o.g. \begin{equation*}
		I_2 = \left[-\frac{\sqrt{2}\sin(k_s \arccos(x))}{k_s\pi}\right]_{\pi}^{-1+\vartheta} = -\frac{\sqrt{2}\sin(k_s \arccos(-1+\vartheta))}{k_s\pi}
	\end{equation*} and \begin{equation*}
		\abs{I_2} \leq \frac{\sqrt{2}\rho}{\pi}.
	\end{equation*} For the last intergral we deduce \begin{equation*}
		I_3 = \frac{1}{\pi} \left( \arcsin(-1+\vartheta) - \arcsin(-1) \right) = \frac{1}{\pi} \arccos(1-\vartheta).
	\end{equation*} The final result follows by \begin{align}
		2^{\abs{M_1}}\left(\frac{\arccos(1-\vartheta)}{\pi}\right)^{\abs{M_1}} \cdot \sqrt{2}^{\abs{M_2}}\left(\frac{\arccos(1-\vartheta)}{\pi}\right)^{\abs{M_2}}  \\ 
		\cdot \left(\frac{\arccos(1-\vartheta)}{\pi}\right)^{\abs{M_3}} \leq 2^{d_s} \sqrt{2}^{2 d_s}  \left(\frac{\arccos(1-\vartheta)}{\pi}\right)^{d}.
	\end{align} The steps work analogously for the second integral.
\end{proof}

\begin{lemma}\label{lem:dis:3}
	Let $\k \in \N_0^d$ with $\norm{\k}{0} \leq d_s \in \D$ and $\vartheta > 0$. Then \begin{equation*}
		\abs{\int_{[-1,-1+\vartheta]^d} T_{\k}^2(\x) \omega(\x) \d\x} \leq 2^{d_s} \left( \frac{\arccos(1-\vartheta)}{\pi} \right)^d
	\end{equation*} and \begin{equation*}
		\abs{\int_{[1-\theta,1]^d} T_{\k}^2(\x) \omega(\x) \d\x} \leq 2^{d_s} \left( \frac{\arccos(1-\vartheta)}{\pi} \right)^d.
	\end{equation*}
\end{lemma}
\begin{proof}
	We define $C(k,x) = \cos(k\arccos(x))$ and write the first integral as the product \begin{align*}
		\int_{[-1,-1+\theta]^d} T_{\k}^2(\x) \omega(\x) \d\x = \prod_{s \in \supp \k} \underbrace{\frac{2}{\pi} \int_{-1}^{-1+\theta} C^2(k_s,x) \frac{1}{\sqrt{1-x^2}} \d x}_{I_1} \\
		 \prod_{ s \in \D\setminus\supp\k } \underbrace{\frac{1}{\pi} \int_{-1}^{-1+\theta} \frac{1}{\sqrt{1-x^2}} \d x}_{I_2}.
	\end{align*} We then have \begin{align*}
		I_1 &= -\frac{2}{\pi} \left[ \frac{2k_s\arccos x + \sin(2k_s\arccos x)}{4 k} \right]_{-1}^{-1+\theta} \\ &= 1 - \frac{\arccos(\theta-1)}{\pi} - \frac{\sin(2 k_s \arccos(\theta-1))}{2 \pi k_s}
	\end{align*} and estimate the absolute value \begin{equation*}
		\abs{I_1} \leq \frac{2\arccos(1-\theta)}{\pi}.
	\end{equation*} The integral $I_2$ was considered in the proof of \cref{lem:dis:1} which results in \begin{equation*}
		\left( \frac{2\arccos(1-\theta)}{\pi} \right)^{d_s} \left( \frac{\arccos(1-\theta)}{\pi} \right)^{d-d_s} = 2^{d_s} \left( \frac{\arccos(1-\theta)}{\pi} \right)^d.
	\end{equation*} The steps work analogously for the second integral.
\end{proof}

Using the previous two lemmas, we are able to prove that the expected value of the matrix $\b H$ from \eqref{matrixH} is close to a diagonal matrix.
\begin{theorem}\label{expectedValue}
	Let $\X \subset [-1+\vartheta,1-\vartheta]^d$, $0 < \vartheta < 1$, be a set of uniformly distributed i.i.d.~nodes, $\F(\X,\I(U))$ the basis matrix for the Chebyshev polynomials with respect to an index set $\I(U)$ of type \eqref{groupedSets} such that $U \subset U_{d_s}$ for superposition threshold $d_s \in \D$, and $\b H$ as in \eqref{matrixH}. Then for the entries of the expected value matrix \begin{equation}\label{Ematrix}
		\b E \coloneqq \mathbb E \left[ \frac{1}{\abs{\X}} \b H \right] \in \R^{\abs{I(U)},\abs{I(U)}}
	\end{equation} we have \begin{equation*}
		\abs{\delta_{\k,\b\ell}-( \b E )_{\k,\b\ell}} \leq  2 \cdot 4^{d_s} \cdot \left( \frac{\arccos(1-\vartheta)}{\pi} \right)^d.
	\end{equation*}
\end{theorem}
\begin{proof}
	The entries of $\b E$ are given by \begin{equation*}
		(\b E)_{\k,\b\ell} = \int_{[-1+\vartheta,1-\vartheta]^d} \omega(\x) T_{\k}(\x) T_{\b\ell}(\x) \d\x.
	\end{equation*} We may rewrite this integral as \begin{align*}
		(\b E)_{\k,\b\ell} = \int_{[-1,1]^d} \omega(\x) T_{\k}(\x) T_{\b\ell}(\x) \d\x - \int_{[-1,-1+\vartheta]^d} \omega(\x) T_{\k}(\x) T_{\b\ell}(\x) \d\x \\ - \int_{[1-\vartheta,1]^d} \omega(\x) T_{\k}(\x) T_{\b\ell}(\x) \d\x \\
		= \delta_{\k,\b\ell} - \int_{[-1,-1+\vartheta]^d} \omega(\x) T_{\k}(\x) T_{\b\ell}(\x) \d\x - \int_{[1-\vartheta,1]^d} \omega(\x) T_{\k}(\x) T_{\b\ell}(\x) \d\x
	\end{align*} and apply \cref{lem:dis:1}, and \cref{lem:dis:3} to obtain the result.
\end{proof}
It remains to consider the eigenvalues of this expected value matrix $\b E$ from \eqref{Ematrix}. \begin{lemma}
	Let $\X \subset [-1+\vartheta,1-\vartheta]^d$, $0 < \vartheta < 1$, be a set of uniformly distributed i.i.d.~nodes, $\F(\X,\I(U))$ the basis matrix for the Chebyshev polynomials with respect to an index set $\I(U)$ of type \eqref{groupedSets} such that $U \subset U_{d_s}$ for superposition threshold $d_s \in \D$, and $\b E$ as in \eqref{Ematrix}. Then for every eigenvalue $\lambda \in \R$ of $\b E$ it holds $$\abs{1-\lambda} < 4^{d_s} \left( \frac{\arccos(1-\vartheta)}{\pi} \right)^d \abs{I(U)}.$$
\end{lemma}
\begin{proof}
	\cref{expectedValue} tells us that we may write $\b E = \b I + \b P$ with the identity matrix $\b I$ and a perturbation matrix $\b P$. Applying the theorem of Bauer and Fike, we immediately obtain the result.
\end{proof}
\begin{corollary}\label{cor:EV}
	Let $\X \subset [-1+\vartheta,1-\vartheta]^d$, $0 < \vartheta < 1$, be a set of uniformly distributed i.i.d.~nodes, $\F(\X,\I(U))$ the basis matrix for the Chebyshev polynomials with respect to an index set $\I(U)$ of type \eqref{groupedSets} such that $U \subset U_{d_s}$ for superposition threshold $d_s \in \D$, and $\b H$ as in \eqref{matrixH}. Then for every eigenvalue $\lambda \in \R$ of $\frac{1}{\abs{\X}} \b H$ we have \begin{equation*}
		\abs{1-\lambda} < \frac{1}{2} + 4^{d_s} \kappa(\delta,\vartheta) \left( \frac{\arccos(1-\vartheta)}{\pi} \right)^d \frac{\abs{\X}}{\log(2\abs{X})}
	\end{equation*} with probability $1 - \delta$ if \begin{equation*}
		\abs{\I(U)} \leq \kappa(\delta,\vartheta) \frac{\abs{\X}}{\log(2\abs{\X})}, \quad\kappa(\delta,\vartheta) \coloneqq \frac{\left(2\theta-\theta^2\right)^{\frac{d}{2}}}{2^{d_s} \cdot 48(\sqrt{2}-\log\delta)}.
	\end{equation*}
\end{corollary}
\begin{proof} 
		We apply the concentration inequality \cite[Proposition 4.1]{KaUlVo19} and note that \begin{equation*}
			M^2 \leq \sup_{\x\in\X} \sum_{\k \in \I(U)} \abs{\sqrt{\omega(\x)} T_{\k}(\x)}^2 \leq \frac{2^{d_s}\abs{\I(U)}}{\sqrt{2\theta-\theta^2}}.
		\end{equation*} Setting $t \coloneqq 0.5$ in the inequality yields the result after rearranging. 
\end{proof}

\cref{cor:EV} now tells us that the singular values $\tau_i$, $i=1,\dots,\abs{I(U)}$, of $\b W \F(\X,\I(U))$ are bounded \begin{equation*}
	\sqrt{\abs{\X}\left(\frac{1}{2} + \gamma\right)} \leq \tau_i \leq \sqrt{\abs{\X}\left(\frac{3}{2} + \gamma\right)}
\end{equation*} with $$ \gamma \coloneqq 4^{d_s} \kappa(\delta,\vartheta) \left( \frac{\arccos(1-\vartheta)}{\pi} \right)^d \frac{\abs{\X}}{\log(2\abs{X})}.$$ Moreover, the norm of the Moore-Penrose inverse is also bounded \begin{equation*}
	\norm{\b{H}^{-1} \F^\top(\X,\I(U)) \b W }{2} < \sqrt{\abs{\X}\left(\frac{1}{2} + \gamma\right)}.
\end{equation*}

\subsection{Active Set Detection}\label{sec:approx:asd}

We determine an initial approximation on the function $f$ by the partial sum $S(\X, \I(U_{d_s})) f (\x)$, cf.\ \eqref{approx_fps}, for a chosen superposition threshold $d_s \in \D$ using the method described in \cref{sec:approx:lsqr}. The set $\I(U_{d_s})$ is formed with the sets \eqref{groupedSets} and order-dependent parameters $N_{\au} \in \N$ such that $N_{\u} = N_{\au}$. In order to understand the structure of $f$, i.e., find the important ANOVA terms $f_{\u}$, we perform a sensitivity analysis using the global sensitivity indices $\gsi{\u}{S(\X, \I(U_{d_s})) f}$. A sensitivity analysis with the indices from the approximation $S(\X, \I(U_{d_s})) f$ of course only makes sense if they behave similarly to the function, i.e.,  the assumption \begin{equation}\label{assumption}
	\gsi{\u_1}{S(\X, \I(U_{d_s})) f} \leq \gsi{\u_2}{S(\X, \I(U_{d_s})) f} \Longrightarrow \gsi{\u_1}{f} \leq \gsi{\u_2}{f}.
\end{equation} holds for $\u_1, \u_2 \in U_{d_s}$. The accuracy of this assumption may depend on multiple factors like the size of the index set $\I(U_{d_s})$, the underlying function and the number of samples, but numerical experiments suggest that we can rely on these indices even for small index sets $\I(U_{d_s})$. We then use a threshold vector $\b\varepsilon \in (0,1)^{d_s}$ to form an active set of ANOVA terms \begin{equation}\label{activeSet}
	U_{\X,\b y}(\b\varepsilon) \coloneqq \emptyset \cup \left\{ \u\subset\D\setminus\emptyset \colon \gsi{\u}{S(\X, \I(U_{d_s})) f} > \varepsilon_{\au} \right\}. 
\end{equation} The inclusion condition \eqref{eq:trunc_inclusion} is fulfilled if we assume that for all $\v \subset \u$ with $\u \in U_{\X,\b y}(\b\varepsilon)$ and $\v \notin U_{\X,\b y}(\b\varepsilon)$ we have $f_{\v} \equiv 0$.

This active set $U_{\X,\b y}(\b\varepsilon)$ is then used to build a corresponding grouped index set $\I(U_{\X,\b y}(\b\varepsilon))$, see \eqref{freq:union}, with finite frequency sets $\I_{\u}$ and parameters $N_{\u}$ as in \eqref{groupedSets}. Depending on the information from the global sensitivity indices one may choose to vary the number of frequencies for terms of the same order, i.e., choose different parameters $N_{\u}$ and $N_{\v}$ for two sets $\u$ and $\v$ with $\au = \av$. 

We obtain the approximate partial sum $S(\X, \I(U_{\X,\b y}(\b\varepsilon))) f (\x)$ by applying the method from \cref{sec:approx:lsqr} again. The benefit of this second approximation is that through the smaller number of ANOVA terms we have a reduced model complexity and we may use more frequencies per remaining ANOVA term in our set $\I(U_{\X,\b y}(\b\varepsilon))$. We obtain the final approximation \begin{equation*}
	f(\x) \approx S(\X, \I(U_{\X,\b y}(\b\varepsilon))) f (\x) \coloneqq \sum_{\bm k \in \I(U_{\X,\b y}(\b\varepsilon))} \hat{f}_{\k} \, \varphi_{\k} (\x). 
\end{equation*} \cref{alg} and \cref{alg2} summarize the proposed method.

\begin{algorithm}[ht]
	\vspace{2mm}
	\begin{tabular}{ l l l }
		\textbf{Input:} & $\X \subset [-1,1]^d$ & finite node set \\
		& & distributed i.i.d.~according to probability density $\omega$ \\ 
		& $\b y$ & evaluation vector \\
		& $d_s \in \D$ & superposition threshold \\
	\end{tabular}
	\begin{algorithmic}[1]
	    \STATE{Choose finite order-dependent parameters $N_i \in \N$, $i=1,2,\dots,d_s$ as in \eqref{groupedSets}.}
	    \STATE{Compute approximation $S(\X, \I(U(d_s))) f$ by solving $$\bfh_{\text{sol}} = (\hat{f}_{\k})_{\k \in \I(U_{d_s})} \leftarrow \argmin_{\bfh} \norm{\b y - \F(\X,\I(U_{d_s})) \bfh}{2}^2.$$}
	    \STATE{Compute global sensitivity indices $\varrho(\u,S(\X, \I(U(d_s))) f)$ for approximation $S(\X, \I(U(d_s))) f$ using \eqref{eq:gsi:formula}.}
        \STATE{Choose threshold vector $\b\varepsilon \in (0,1)^{d_s}$ and build active set $U_{\X,\b y}(\b\varepsilon)$ according to \eqref{activeSet}.}
        \STATE{Use information from global sensitivity indices to choose parameters $N_{\u} \in \N$ for every ANOVA term in $U_{\X,\b y}(\b\varepsilon)$ to obtain $\I(U_{\X,\b y}(\b\varepsilon))$.}
	    \STATE{Compute approximation $S(\X, \I(U_{\X,\b y}(\b\varepsilon))) f$ by solving $$\bfh_{\text{sol}} = (\hat{f}_{\k})_{\k \in \I(U_{\X,\b y}(\b\varepsilon))} \leftarrow \argmin_{\bfh} \norm{\b y - \F(\X,\I(U_{\X,\b y}(\b\varepsilon))) \bfh}{2}^2.$$}
	\end{algorithmic}
	\begin{tabular}{ l l l }
		\textbf{Output:} & $\hat{f}_{\k} \in \R, \k \in \I(U_{\X,\b y}(\b\varepsilon))$ & approximations to basis coefficients $\fc{k}$ \\
	\end{tabular}
	\caption{ANOVA Approximation Method with nodes distributed according to the density $\omega$}
	\label{alg}
\end{algorithm}

\begin{algorithm}[ht]
	\vspace{2mm}
	\begin{tabular}{ l l l }
		\textbf{Input:} & $\X \subset [-1,1]^d$ & finite node set \\
		& & distributed uniformly \\
		& $\b y$ & evaluation vector \\
		& $d_s \in \D$ & superposition threshold \\
		& $\theta \in (0,1)$ & padding parameter \\
	\end{tabular}
	\begin{algorithmic}[1]
		\STATE{Choose finite order-dependent parameters $N_i \in \N$, $i=1,2,\dots,d_s$ as in \eqref{groupedSets}.}
		\STATE{Scale nodes $\X$ into interval $[-1+\vartheta,1-\vartheta]^d$.}
		\STATE{Set $\b W = \mathrm{diag}(\b w)$ with $\b w = (\omega(\x))_{\x\in \X}$.}
		\STATE{Compute approximation $S(\X, \I(U(d_s))) f$ by solving $$\bfh_{\text{sol}} = (\hat{f}_{\k})_{\k \in \I(U_{d_s})} \leftarrow \argmin_{\bfh} \norm{\b y - \F(\X,\I(U_{d_s})) \bfh}{2, \b W}^2.$$}
		\STATE{Compute global sensitivity indices $\varrho(\u,S(\X, \I(U(d_s))) f)$ for approximation $S(\X, \I(U(d_s))) f$ using \eqref{eq:gsi:formula}.}
		\STATE{Choose threshold vector parameter $\b\varepsilon > 0$ and build active set $U_{\X,\b y}(\b\varepsilon)$ according to \eqref{activeSet}.}
		\STATE{Use information from global sensitivity indices to choose parameters $N_{\u} \in \N$ for every ANOVA term in $U_{\X,\b y}(\b\varepsilon)$ to obtain $\I(U_{\X,\b y}(\b\varepsilon))$.}
		\STATE{Compute approximation $S(\X, \I(U_{\X,\b y}(\b\varepsilon))) f$ by solving $$\bfh_{\text{sol}} = (\hat{f}_{\k})_{\k \in \I(U_{\X,\b y}(\b\varepsilon))} \leftarrow \argmin_{\bfh} \norm{\b y - \F(\X,\I(U_{\X,\b y}(\b\varepsilon))) \bfh}{2, \b W}^2.$$}
	\end{algorithmic}
	\begin{tabular}{ l l l }
		\textbf{Output:} & $\hat{f}_{\k} \in \R, \k \in \I(U_{\X,\b y}(\b\varepsilon))$ & approximations to basis coefficients $\fc{k}$ \\
	\end{tabular}
	\caption{ANOVA Approximation Method with nodes distributed uniformly}
	\label{alg2}
\end{algorithm}

Using the iterative least-squares method LSQR, cf.~\cite{PaSa82}, the arithmetic complexity of an iteration is determined by the matrix-vector multiplications. The following results show the precise complexity of one iteration if we focus on the Chebyshev system.

\begin{theorem}\label{lem:complexity}
	Let $\Lomd$ be the weighted Lebesgue space with Chebyshev product weight $\omega$ and the Chebyshev system $\{T_{\k}\}_{\k \in \N_0^d}$ as orthonormal basis. Moreover, let $\I(U)$ for $U \subset \mathcal{P}(\D)$ be formed with the sets \eqref{groupedSets} and parameters $N_{\u} \in \N$, $\u \in U$. Then each iteration of the LSQR algorithm to solve the minimization problem \eqref{min_problem} or \eqref{min_problem_uni} with node set $\X$ and evaluations $\b y \in \R^{\abs{\X}}$ has a complexity in $$\mathcal{O}\left( \sum_{\u\in U} N_{\u}^{\au} \log N_{\u} + \abs{\X} \abs{U} \right).$$
\end{theorem} 
\begin{proof}
	During each iteration of the least-squares algorithm \cite{PaSa82} there are two matrix multiplications, one with $\b F$ and one with $\b{F}^\top$. We have to compute one nonequispaced fast cosine transform (NFCT) and one adjoined nonequispaced fast cosine transform for each ANOVA term $f_{\u}$, $\u \in U$, with complexity in $\mathcal{O}(N_{\u}^{\au} \log N_{\u})$ each. Summing over the complexities yields the result.
\end{proof}

\begin{corollary}\label{cor:comp}
	Let $\Lomd$ be the weighted Lebesgue space with Chebyshev product weight $\omega$ and the Chebyshev system $\{T_{\k}\}_{\k \in \N_0^d}$ as orthonormal basis. Moreover, let $\I(U)$ for $U = U_{d_s} \subset \mathcal{P}(\D)$ with superposition threshold $d_s \in \D$ be formed with the sets \eqref{groupedSets} and order-dependent parameters $N_{\u} = N_{\au} \in \N$, $\u \in U$. Then each iteration of the algorithm to solve the minimization problem \eqref{min_problem} or \eqref{min_problem_uni} with node set $\X$ and evaluations $\b y \in \R^{\abs{\X}}$ has a complexity in $$\mathcal{O}\left( d^{d_s} \left( N_{d_s}^{d_s} \log N_{d_s} + \abs{\X} \right)\right)$$ if $ N_{d_s} = \max_{j=1,2,\dots,d_s}  N_{j}$.	
\end{corollary}
\begin{proof}
	This follows directly from the previous theorem and the estimate $\abs{U_{d_s}}\leq(\e\cdot d/d_s)^{d_s}$ from \eqref{polyterms}.
\end{proof}

\begin{remark}
    The proposed method is in principle related to sparse polynomial approximation, see e.g.~\cite{ChCoSchw15}. The first step of considering ANOVA terms of order up to the superposition dimension $d_s$ is equal to considering the basis functions $\phi_{\k}$ with $\norm{\k}{0} \leq d_s$. We combine this with fast algorithms for the solution of the corresponding least-squares problems that are able to deal with scattered data. Our approach also differs in the fact that we use the importance of ANOVA terms with global sensitivity indices to characterize important basis functions.
\end{remark}

\section{Numerical Experiments}\label{sec:numerics}

In this section we apply the proposed approximation method to high-dimension benchmark functions. We start with an $8$-dimensional function that is the sum of products of B-splines in \cref{sec:splines}. A similar function has been considered in \cite{PoVo17}. In \cref{sec:friedman} we consider the well-known Friedman benchmark functions which have previously been used an example for a synthetic regression problem, cf.\ \cite{MeLeHo03, BeGaMo09, BiDaLa11, PoSchm21}. The method has been implemented as a Julia package \cite{BaSchm20}. The padding parameter for uniformly distributed nodes is fixed as $\theta = 10^{-4}$.

\subsection{B-Spline Function}\label{sec:splines}

We apply our method to the test function $\fun{f}{[-1,1]^8}{\R}$, \begin{equation}\label{testfun}
	f(\x) = B_2(x_1) B_4(x_5) + B_2(x_2) B_4(x_6) + B_2(x_3) B_4(x_7) + B_2(x_4) B_4(x_8),
\end{equation}
where $B_2$ and $B_4$ are parts of shifted, scaled and dilated B-splines of order 2 and 4, respectively. In \Cref{fig:splines} we have visualized the splines $B_2$ and $B_4$ which are elements of $\Lom$ with weight $\tilde\omega (x):=\pi^{-1} \cdot (1-x^2)^{-1/2}$ such that $\norm{B_2}{\Lom} = \norm{B_4}{\Lom} = 1$. We remark that the basis coefficients $\mathrm{c}_{k}\!(B_2)$ and $\mathrm{c}_{k}\!(B_4)$ decay like $\sim k^{-3}$ and $\sim k^{-5}$, respectively. Moreover, $f$ is an element of the tensor product space $\Lomdn$. As basis we have the normed Chebyshev polynomials of first kind $\{T_{\k}\}_{\k\in\N_0^8}$. 

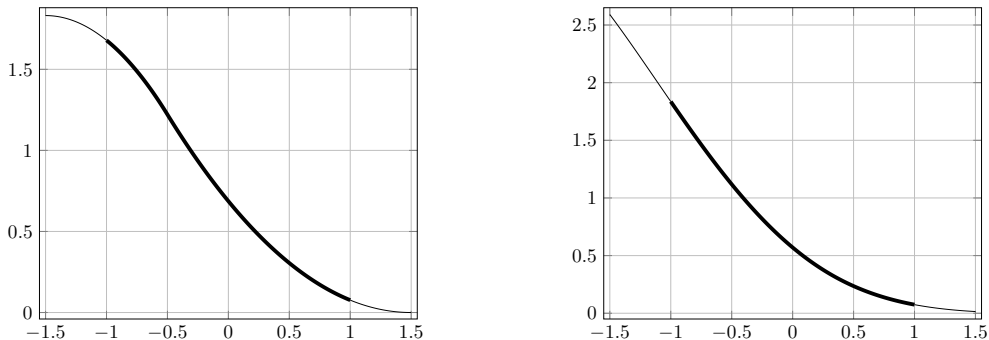
\begin{figure}[ht]
	\begin{minipage}{0.45\textwidth}
		
		\begin{tikzpicture}[scale=0.725]
		\begin{axis}[
		grid=major,
		xmin=-1.55,
		xmax=1.55,
		ymin=-0.04,
		ymax=1.88
		]
		\addplot[smooth] coordinates {
			(-1.5,1.8315970148873788)(-1.49,1.831535961653549)(-1.48,1.8313528019520602)(-1.47,1.8310475357829126)(-1.46,1.8306201631461054)(-1.45,1.830070684041639)(-1.44,1.829399098469514)(-1.43,1.8286054064297295)(-1.42,1.8276896079222855)(-1.41,1.8266517029471827)(-1.4,1.8254916915044204)(-1.39,1.8242095735939998)(-1.38,1.822805349215919)(-1.37,1.82127901837018)(-1.36,1.8196305810567808)(-1.35,1.8178600372757237)(-1.34,1.8159673870270066)(-1.33,1.8139526303106301)(-1.32,1.811815767126595)(-1.31,1.8095567974749005)(-1.3,1.807175721355547)(-1.29,1.804672538768534)(-1.28,1.8020472497138622)(-1.27,1.7992998541915317)(-1.26,1.796430352201541)(-1.25,1.7934387437438917)(-1.24,1.7903250288185828)(-1.23,1.787089207425615)(-1.22,1.7837312795649887)(-1.21,1.7802512452367023)(-1.2,1.776649104440757)(-1.19,1.772924857177153)(-1.18,1.7690785034458896)(-1.17,1.7651100432469666)(-1.16,1.7610194765803846)(-1.15,1.7568068034461437)(-1.14,1.752472023844244)(-1.13,1.7480151377746849)(-1.12,1.743436145237466)(-1.11,1.7387350462325886)(-1.1,1.7339118407600522)(-1.09,1.728966528819856)(-1.08,1.7238991104120007)(-1.07,1.7187095855364867)(-1.06,1.713397954193313)(-1.05,1.7079642163824809)(-1.04,1.702408372103989)(-1.03,1.6967304213578378)(-1.02,1.690930364144028)(-1.01,1.6850082004625588)(-1.0,1.6789639303134305)
		};
		\addplot[smooth, line width=2pt] coordinates {
			
			(-1.0,1.6789639303134305)(-0.99,1.6727975536966426)(-0.98,1.6665090706121963)(-0.97,1.6600984810600905)(-0.96,1.6535657850403256)(-0.95,1.6469109825529011)(-0.94,1.640134073597818)(-0.93,1.6332350581750756)(-0.92,1.6262139362846741)(-0.91,1.6190707079266131)(-0.9,1.6118053731008932)(-0.89,1.604417931807514)(-0.88,1.596908384046476)(-0.87,1.5892767298177786)(-0.86,1.581522969121422)(-0.85,1.573647101957406)(-0.84,1.5656491283257312)(-0.83,1.5575290482263973)(-0.82,1.5492868616594042)(-0.81,1.5409225686247519)(-0.8,1.5324361691224404)(-0.79,1.5238276631524694)(-0.78,1.5150970507148396)(-0.77,1.5062443318095509)(-0.76,1.4972695064366026)(-0.75,1.4881725745959953)(-0.74,1.4789535362877284)(-0.73,1.469612391511803)(-0.72,1.4601491402682185)(-0.71,1.4505637825569742)(-0.7,1.4408563183780712)(-0.69,1.431026747731509)(-0.68,1.4210750706172874)(-0.67,1.4110012870354072)(-0.66,1.4008053969858671)(-0.65,1.3904874004686685)(-0.64,1.38004729748381)(-0.63,1.3694850880312932)(-0.62,1.3588007721111166)(-0.61,1.3479943497232811)(-0.6,1.3370658208677864)(-0.59,1.3260151855446325)(-0.58,1.3148424437538195)(-0.57,1.3035475954953475)(-0.56,1.2921306407692161)(-0.55,1.2805915795754257)(-0.54,1.2689304119139762)(-0.53,1.257147137784867)(-0.52,1.2452417571880992)(-0.51,1.2332142701236721)(-0.5,1.2210646765915858)(-0.49,1.2088845564425847)(-0.48,1.1967654895274131)(-0.47,1.1847074758460714)(-0.46,1.1727105153985589)(-0.45,1.1607746081848762)(-0.44,1.1488997542050232)(-0.43,1.1370859534589994)(-0.42,1.1253332059468055)(-0.41,1.113641511668441)(-0.4,1.1020108706239062)(-0.39,1.0904412828132009)(-0.38,1.0789327482363253)(-0.37,1.067485266893279)(-0.36,1.0560988387840626)(-0.35,1.0447734639086756)(-0.34,1.0335091422671183)(-0.33,1.0223058738593906)(-0.32,1.0111636586854922)(-0.31,1.0000824967454234)(-0.3,0.9890623880391844)(-0.29,0.978103332566775)(-0.28,0.967205330328195)(-0.27,0.9563683813234448)(-0.26,0.9455924855525242)(-0.25,0.9348776430154329)(-0.24,0.9242238537121712)(-0.23,0.9136311176427393)(-0.22,0.903099434807137)(-0.21,0.8926288052053639)(-0.2,0.8822192288374208)(-0.19,0.871870705703307)(-0.18,0.861583235803023)(-0.17,0.8513568191365684)(-0.16,0.8411914557039434)(-0.15,0.8310871455051482)(-0.14,0.8210438885401823)(-0.13,0.8110616848090462)(-0.12,0.8011405343117395)(-0.11,0.7912804370482625)(-0.1,0.7814813930186147)(-0.09,0.7717434022227971)(-0.08,0.7620664646608089)(-0.07,0.75245058033265)(-0.06,0.7428957492383209)(-0.05,0.7334019713778211)(-0.04,0.7239692467511512)(-0.03,0.7145975753583108)(-0.02,0.7052869571993)(-0.01,0.6960373922741186)(0.0,0.686848880582767)(0.01,0.677721422125245)(0.02,0.6686550169015524)(0.03,0.6596496649116895)(0.04,0.650705366155656)(0.05,0.6418221206334523)(0.06,0.632999928345078)(0.07,0.6242387892905334)(0.08,0.6155387034698184)(0.09,0.606899670882933)(0.1,0.598321691529877)(0.11,0.5898047654106507)(0.12,0.581348892525254)(0.13,0.5729540728736868)(0.14,0.5646203064559493)(0.15,0.5563475932720413)(0.16,0.5481359333219629)(0.17,0.539985326605714)(0.18,0.5318957731232947)(0.19,0.5238672728747051)(0.2,0.515899825859945)(0.21,0.5079934320790145)(0.22,0.5001480915319135)(0.23,0.4923638042186422)(0.24,0.4846405701392004)(0.25,0.4769783892935882)(0.26,0.4693772616818056)(0.27,0.46183718730385254)(0.28,0.45435816615972907)(0.29,0.44694019824943526)(0.3,0.43958328357297094)(0.31,0.4322874221303362)(0.32,0.42505261392153104)(0.33,0.41787885894655546)(0.34,0.41076615720540943)(0.35,0.4037145086980931)(0.36,0.39672391342460617)(0.37,0.38979437138494905)(0.38,0.3829258825791213)(0.39,0.3761184470071232)(0.4,0.36937206466895467)(0.41,0.3626867355646157)(0.42,0.3560624596941064)(0.43,0.34949923705742664)(0.44,0.34299706765457644)(0.45,0.33655595148555584)(0.46,0.3301758885503648)(0.47,0.3238568788490034)(0.48,0.31759892238147147)(0.49,0.31140201914776916)(0.5,0.30526616914789645)(0.51,0.2991913723818533)(0.52,0.2931776288496397)(0.53,0.28722493855125575)(0.54,0.2813333014867014)(0.55,0.2755027176559765)(0.56,0.2697331870590813)(0.57,0.26402470969601566)(0.58,0.25837728556677964)(0.59,0.25279091467137305)(0.6,0.24726559700979617)(0.61,0.24180133258204875)(0.62,0.23639812138813107)(0.63,0.2310559634280428)(0.64,0.22577485870178424)(0.65,0.22055480720935516)(0.66,0.21539580895075575)(0.67,0.21029786392598582)(0.68,0.2052609721350456)(0.69,0.2002851335779349)(0.7,0.19537034825465383)(0.71,0.1905166161652022)(0.72,0.18572393730958014)(0.73,0.1809923116877878)(0.74,0.17632173929982506)(0.75,0.17171222014569176)(0.76,0.167163754225388)(0.77,0.162676341538914)(0.78,0.15824998208626959)(0.79,0.1538846758674546)(0.8,0.1495804228824692)(0.81,0.1453372231313135)(0.82,0.1411550766139873)(0.83,0.13703398333049077)(0.84,0.13297394328082365)(0.85,0.12897495646498627)(0.86,0.12503702288297833)(0.87,0.12116014253480015)(0.88,0.11734431542045135)(0.89,0.1135895415399323)(0.9,0.10989582089324268)(0.91,0.10626315348038277)(0.92,0.1026915393013523)(0.93,0.09918097835615158)(0.94,0.09573147064478035)(0.95,0.09234301616723878)(0.96,0.08901561492352662)(0.97,0.08574926691364405)(0.98,0.0825439721375912)(0.99,0.07939973059536794)(1.0,0.07631654228697411)
		};
		\addplot[smooth] coordinates {
			(1.0,0.07631654228697411)(1.01,0.07329440721240986)(1.02,0.07033332537167532)(1.03,0.06743329676477036)(1.04,0.0645943213916949)(1.05,0.06181639925244893)(1.06,0.059099530347032754)(1.07,0.05644371467544607)(1.08,0.0538489522376889)(1.09,0.0513152430337613)(1.1,0.04884258706366341)(1.11,0.046430984327395085)(1.12,0.0440804348249562)(1.13,0.04179093855634703)(1.14,0.03956249552156736)(1.15,0.037395105720617335)(1.16,0.03528876915349694)(1.17,0.03324348582020592)(1.18,0.03125925572074454)(1.19,0.02933607885511287)(1.2,0.027473955223310774)(1.21,0.02567288482533811)(1.22,0.023932867661195022)(1.23,0.022253903730881642)(1.24,0.0206359930343979)(1.25,0.019079135571743528)(1.26,0.017583331342918798)(1.27,0.016148580347923706)(1.28,0.014774882586758256)(1.29,0.013462238059422243)(1.3,0.012210646765915802)(1.31,0.011020108706239068)(1.32,0.009890623880391908)(1.33,0.008822192288374185)(1.34,0.007814813930186102)(1.35,0.0068684888058275235)(1.36,0.005983216915298857)(1.37,0.005158998258599424)(1.38,0.0043958328357296995)(1.39,0.003693720646689479)(1.4,0.0030526616914791026)(1.41,0.0024726559700980278)(1.42,0.001953703482546593)(1.43,0.0014958042288245957)(1.44,0.0010989582089323063)(1.45,0.0007631654228698604)(1.46,0.0004884258706365805)(1.47,0.00027473955223300875)(1.48,0.00012210646765928068)(1.49,3.0526616914854066e-5)(1.5,0.0)
		};
		
		\end{axis}
		\end{tikzpicture}
		
	\end{minipage}
	\hfill
	\begin{minipage}{0.45\textwidth}
		
		\begin{tikzpicture}[scale=0.725]
		\begin{axis}[
		grid=major,
		xmin=-1.55,
		xmax=1.55,
		ymin=-0.05,
		ymax=2.65
		]
		\addplot[smooth] coordinates {
			(-1.5,2.5907157498280107)(-1.49,2.576549597759345)(-1.48,2.56231490123993)(-1.47,2.548013758749521)(-1.46,2.533648254636701)(-1.45,2.519220459118872)(-1.44,2.5047324282822614)(-1.43,2.49018620408192)(-1.42,2.4755838143417197)(-1.41,2.4609272727543576)(-1.4,2.4462185788813535)(-1.39,2.43145971815305)(-1.38,2.4166526618686133)(-1.37,2.4017993671960327)(-1.36,2.3869017771721204)(-1.35,2.371961820702512)(-1.34,2.356981412561666)(-1.33,2.3419624533928642)(-1.32,2.3269068297082125)(-1.31,2.3118164138886383)(-1.3,2.2966930641838936)(-1.29,2.2815386247125535)(-1.28,2.266354925462014)(-1.27,2.2511437822884983)(-1.26,2.2359069969170493)(-1.25,2.2206463569415344)(-1.24,2.205363635824645)(-1.23,2.1900605928978942)(-1.22,2.1747389733616185)(-1.21,2.1594005082849796)(-1.2,2.144046914605959)(-1.19,2.1286798951313646)(-1.18,2.113301138536825)(-1.17,2.0979123193667935)(-1.16,2.0825150980345457)(-1.15,2.0671111208221817)(-1.14,2.051702019880623)(-1.13,2.036289413229616)(-1.12,2.0208749047577292)(-1.11,2.005460084222354)(-1.1,1.9900465272497054)(-1.09,1.9746357953348228)(-1.08,1.959229435841567)(-1.07,1.9438289820026224)(-1.06,1.9284359529194974)(-1.05,1.9130518535625227)(-1.04,1.897678174770853)(-1.03,1.8823163932524651)(-1.02,1.8669679715841596)(-1.01,1.851634358211561)(-1.0,1.8363169874491154)
		};
		\addplot[smooth, line width=2pt] coordinates {
			(-1.0,1.8363169874491154)(-0.99,1.8210172794800938)(-0.98,1.8057366403565887)(-0.97,1.7904764619995164)(-0.96,1.775238122198618)(-0.95,1.7600229846124547)(-0.94,1.7448323987684131)(-0.93,1.7296677000627034)(-0.92,1.714530209760357)(-0.91,1.6994212349952298)(-0.9,1.6843420687700001)(-0.89,1.6692939899561707)(-0.88,1.654278263294066)(-0.87,1.639296139392835)(-0.86,1.6243488547304485)(-0.85,1.6094376316537016)(-0.84,1.5945636783782122)(-0.83,1.579728188988421)(-0.82,1.5649323434375926)(-0.81,1.550177307547815)(-0.8,1.5354642330099975)(-0.79,1.5207942573838746)(-0.78,1.506168504098003)(-0.77,1.4915880824497634)(-0.76,1.4770540876053588)(-0.75,1.4625676005998154)(-0.74,1.4481296883369836)(-0.73,1.4337414035895355)(-0.72,1.4194037849989676)(-0.71,1.4051178570755996)(-0.7,1.3908846301985727)(-0.69,1.3767051006158535)(-0.68,1.362580250444231)(-0.67,1.3485110476693163)(-0.66,1.3344984461455454)(-0.65,1.3205433855961757)(-0.64,1.3066467916132898)(-0.63,1.2928095756577915)(-0.62,1.2790326350594092)(-0.61,1.2653168530166938)(-0.6,1.2516630985970199)(-0.59,1.2380722267365842)(-0.58,1.2245450782404081)(-0.57,1.211082479782335)(-0.56,1.1976852439050323)(-0.55,1.1843541690199895)(-0.54,1.1710900394075208)(-0.53,1.1578936252167618)(-0.52,1.144765682465673)(-0.51,1.1317069530410364)(-0.5,1.1187181646984592)(-0.49,1.1058000310623701)(-0.48,1.0929532516260214)(-0.47,1.080178511751489)(-0.46,1.0674764826696719)(-0.45,1.0548478214802917)(-0.44,1.042293171151894)(-0.43,1.0298131605218463)(-0.42,1.0174084042963412)(-0.41,1.005079503050393)(-0.4,0.9928270432278394)(-0.39,0.980651597141342)(-0.38,0.9685537229723844)(-0.37,0.9565339647712748)(-0.36,0.9445928524571434)(-0.35,0.932730901817944)(-0.34,0.920948614510454)(-0.33,0.9092464780602731)(-0.32,0.8976249658618252)(-0.31,0.8860845371783562)(-0.3,0.8746256371419363)(-0.29,0.8632486967534586)(-0.28,0.8519541328826388)(-0.27,0.8407423482680163)(-0.26,0.8296137315169536)(-0.25,0.8185686571056365)(-0.24,0.8076074853790738)(-0.23,0.7967305625510973)(-0.22,0.7859382207043621)(-0.21,0.7752307777903472)(-0.2,0.7646085376293537)(-0.19,0.7540717899105065)(-0.18,0.7436208101917535)(-0.17,0.7332558598998657)(-0.16,0.7229771863304376)(-0.15,0.7127850226478865)(-0.14,0.7026795878854531)(-0.13,0.6926610869452016)(-0.12,0.6827297105980186)(-0.11,0.6728856354836145)(-0.1,0.6631290241105224)(-0.09,0.6534600248560993)(-0.08,0.643878771966525)(-0.07,0.634385385556802)(-0.06,0.6249799716107568)(-0.05,0.6156626219810387)(-0.04,0.60643341438912)(-0.03,0.5972924124252964)(-0.02,0.5882396655486869)(-0.01,0.5792752090872334)(0.0,0.5703990642377013)(0.01,0.5616112380656788)(0.02,0.5529117235055776)(0.03,0.5443004993606325)(0.04,0.5357775303029014)(0.05,0.5273427668732655)(0.06,0.5189961454814291)(0.07,0.5107375884059197)(0.08,0.502567003794088)(0.09,0.4944842856621079)(0.1,0.48648931389497624)(0.11,0.47858195424651356)(0.12,0.4707620583393631)(0.13,0.46302946366499154)(0.14,0.4553839935836885)(0.15,0.4478254573245672)(0.16,0.4403536499855636)(0.17,0.43296835253343713)(0.18,0.42566933180377026)(0.19,0.4184563405009688)(0.2,0.4113291171982614)(0.21,0.4042873863377002)(0.22,0.3973308582301606)(0.23,0.39045922905534103)(0.24,0.3836721808617629)(0.25,0.37696938156677107)(0.26,0.3703504849565336)(0.27,0.3638151306860416)(0.28,0.3573629442791095)(0.29,0.35099353712837467)(0.3,0.3447065064952979)(0.31,0.33850143551016304)(0.32,0.3323778931720772)(0.33,0.3263354343489707)(0.34,0.3203735997775969)(0.35,0.3144919160635324)(0.36,0.30868989568117705)(0.37,0.3029670369737538)(0.38,0.29732282415330896)(0.39,0.29175672730071167)(0.4,0.2862682023656547)(0.41,0.2808566911666535)(0.42,0.27552162139104713)(0.43,0.2702624065949978)(0.44,0.2650784462034904)(0.45,0.25996912551033385)(0.46,0.2549338156781593)(0.47,0.24997187373842208)(0.48,0.24508264259139978)(0.49,0.24026545100619365)(0.5,0.23551961362072826)(0.51,0.23084443167774973)(0.52,0.2262392011208132)(0.53,0.22170322069026907)(0.54,0.21723579265926202)(0.55,0.2128362228337309)(0.56,0.20850382055240876)(0.57,0.20423789868682296)(0.58,0.20003777364129466)(0.59,0.1959027653529399)(0.6,0.19183219729166825)(0.61,0.1878253964601841)(0.62,0.18388169339398538)(0.63,0.18000042216136491)(0.64,0.1761809203634091)(0.65,0.17242252913399914)(0.66,0.16872459313980978)(0.67,0.16508646058031068)(0.68,0.16150748318776503)(0.69,0.15798701622723083)(0.7,0.15452441849655982)(0.71,0.15111905232639816)(0.72,0.14777028358018615)(0.73,0.14447748165415839)(0.74,0.14124001947734355)(0.75,0.13805727351156458)(0.76,0.13492862375143866)(0.77,0.1318534537243771)(0.78,0.1288311504905855)(0.79,0.12586110464306355)(0.8,0.12294271030760524)(0.81,0.12007536514279879)(0.82,0.11725847034002655)(0.83,0.11449143062346492)(0.84,0.11177365425008497)(0.85,0.10910455300965141)(0.86,0.10648354222472366)(0.87,0.10391004075065488)(0.88,0.10138347097559292)(0.89,0.09890325882047932)(0.9,0.0964688337390503)(0.91,0.0940796287178359)(0.92,0.0917350802761607)(0.93,0.08943462846614314)(0.94,0.08717771687269621)(0.95,0.0849637926135268)(0.96,0.08279230633913619)(0.97,0.08066271223281983)(0.98,0.07857446801066734)(0.99,0.07652703492156256)(1.0,0.07451987774718355)
		};
		\addplot[smooth] coordinates {
			(1.0,0.07451987774718355)(1.01,0.07255246480200256)(1.02,0.07062426793328602)(1.03,0.06873476252109463)(1.04,0.06688342747828323)(1.05,0.0650697452505009)(1.06,0.06329320181619089)(1.07,0.06155328668659069)(1.08,0.05984949290573201)(1.09,0.05818131705044068)(1.1,0.05654825923033684)(1.11,0.054949823087834755)(1.12,0.05338551579814294)(1.13,0.05185484806926412)(1.14,0.05035733414199514)(1.15,0.04889249178992714)(1.16,0.047459842319445444)(1.17,0.046058910569729564)(1.18,0.044689224912753234)(1.19,0.043350317253284386)(1.2,0.04204172302888513)(1.21,0.04076298120991183)(1.22,0.039513634299515)(1.23,0.03829322833363942)(1.24,0.03710131288102401)(1.25,0.03593744104320194)(1.26,0.03480116945450057)(1.27,0.03369205828204146)(1.28,0.03260967122574037)(1.29,0.03155357551830729)(1.3,0.03052334192524638)(1.31,0.029518544744856025)(1.32,0.028538761808228806)(1.33,0.02758357447925152)(1.34,0.02665256765460516)(1.35,0.02574532976376492)(1.36,0.024861452769000203)(1.37,0.024000532165374615)(1.38,0.023162166980745982)(1.39,0.02234595977576629)(1.4,0.02155151664388177)(1.41,0.020778447211332852)(1.42,0.020026364637154154)(1.43,0.01929488561317452)(1.44,0.01858363036401698)(1.45,0.017892222647098773)(1.46,0.01722028975263135)(1.47,0.016567462503620355)(1.48,0.015933375255865644)(1.49,0.015317665897961278)(1.5,0.014719975851295516)
		};
		\end{axis}
		\end{tikzpicture}
		
	\end{minipage}
	
	\caption{B-splines $B_2$ (left) and $B_4$ (right). The relevant part in $[-1,1]$ is highlighted.}
	\label{fig:splines}
\end{figure}

The ANOVA terms $f_{\u}$ are nonzero for $$\u \in U^\ast \coloneqq \mathcal{P}(\{ 1,5 \}) \cup \mathcal{P}(\{ 2,6 \}) \cup \mathcal{P}(\{ 3,7 \})\cup \mathcal{P}(\{ 4,8 \})$$ which we call our active set of terms. The function $f$ therefore has a superposition dimension $2$ for the accuracy $\delta = 1$, cf.\ \eqref{eq:anova:superposition}, i.e., $\mathrm{T}_{3} f = f$. This leads to $d_s = 2$ being the optimal choice for the superposition threshold with no error caused by ANOVA truncation. In a scattered data scenario with an unknown function $f$ this information is of course not known. For our numerical experiments we fix two sampling sets, $\X_{\text{uni}} \subset (-1,1)^9$ with uniformly distributed nodes and $\X_{\text{cheb}} \subset [-1,1]^9$ with nodes distributed according to $\omega$. Moreover, we have $M \coloneqq \abs{\X_{\text{cheb}}} = \abs{\X_{\text{uni}}} = 10000$, and an evaluation vector $\b y = ( f(\x) )_{\x\in \X}$. In the following, we always choose the superposition threshold $d_s = 2$.

Our first goal is to detect the ANOVA terms in $U^\ast$. To this end, we use the first step of our method and choose a frequency index set $\I(U_{d_s}) \subset \N_0^8$ through order-dependent sets $\I_\emptyset = \{\b 0\}$, $$\I_{\u} = \{\k\in\N^d_0 \colon \k_{\uc} = \b 0, k_{j} = 1,\dots,N_{\au}-1, j\in\u \}$$ with $N_1,N_2 \in \N$. We consider the two errors \begin{equation*}
	\varepsilon_{\ell_2}^{\X}(f,\tilde f) = \frac{1}{\norm{\b y}{2}}\sqrt{\sum_{\x\in\X}\abs{f(\x)-\tilde{f}(\x)}^2}  
\end{equation*} and $$ \varepsilon_{\mathrm{L}_2}(f,\tilde f) = \frac{1}{\norm{f}{\Lomdn}} \norm{f - \tilde f}{\Lomdn}$$
where $\tilde{f}$ is an approximation on $f$. Here, the error $\varepsilon_{\ell_2}^{\X}$ can be regarded as a training error since it is taken at a given sampling set $\X$ and the error $\varepsilon_{\mathrm{L}_2}$ as a generalization error since it measures the error in the basis coefficients. 

Our goal is to find the important ANOVA terms, i.e., the terms in $U^\ast$. In order to achieve this we expect to have intervals $I_j \subset (0,1)$, $j=1,2$, in which to choose the threshold vector $\b\varepsilon$ with $\eps_j \in I_j$ such that \begin{equation*}
 	U_{\X,\b y}(\b\varepsilon) = U^\ast.
 \end{equation*}The results of our numerical experiments with the function $f$ from \eqref{testfun} are displayed in \Cref{tab:res_det}.

\begin{remark}
	The norm occuring in the error $\varepsilon_{\mathrm{L}_2}$ can be calculated using Parselval's identity \begin{align*}
		\norm{f - S(\I(U),\X) f}{\Lomd}^2 &= \norm{f}{\Lomd}^2 + \sum_{\k \in \I(U)} \abs{ \fc{\k} - \hat{f}_{\k} }^2 \\ &\quad - \sum_{\k \in \I(U)} \abs{ \fc{\k} }^2.
	\end{align*} This is of course only possible if the original coefficients and the norm of the function $f$ is known.
\end{remark}

\begin{table*}[ht]\centering
\begin{tabular}{@{}rrrrrrrrrrrrr@{}}\toprule \multicolumn{3}{c}{size of index sets} & \phantom{i}& \multicolumn{2}{c}{relative errors $\X_{\text{cheb}}$} & \phantom{i} &\multicolumn{2}{c}{relative errors $\X_{\text{uni}}$} \\\cmidrule{1-3} \cmidrule{5-6} \cmidrule{8-9}  $N_1$ & $N_2$ &  $\abs{\I(U_{2})}$ && $\varepsilon_{\ell_2}^{\X_{\text{cheb}}}(f,\tilde{f}_1)$ & $\varepsilon_{\mathrm{L}_2}(f,\tilde{f}_1)$  && $\varepsilon_{\ell_2}^{\X_{\text{uni}}}(f,\tilde{f}_2)$ & $\varepsilon_{\mathrm{L}_2}(f,\tilde{f}_2)$  \\\midrule 
		20 & 8 & 1525 && $5.1 \cdot {10}^{-4}$ & $6.9 \cdot {10}^{-4}$ && $5.3 \cdot {10}^{-4}$ & $8.9 \cdot {10}^{-4}$ \\ 
		20 & 12 & 3541 && $1.5 \cdot {10}^{-4}$ & $4.1 \cdot {10}^{-4}$ && $3.2 \cdot {10}^{-4}$ & $5.1\cdot {10}^{-3}$ \\
		20 & 16 & 6453 && $6.8 \cdot {10}^{-5}$ & $3.9 \cdot {10}^{-4}$ && $2.8 \cdot {10}^{-3}$ & $2.6 \cdot {10}^{-1}$ \\ 
		20 & 20 & 10261 && $3.3 \cdot {10}^{-3}$ & $1.6 \cdot {10}^{-1}$ && $2.8 \cdot {10}^{-3}$ & $5.4 \cdot {10}^{-1}$ \\ \midrule
		40 & 8 & 1685 && $5.0 \cdot {10}^{-4}$ & $6.9 \cdot {10}^{-4}$ && $5.2 \cdot {10}^{-4}$ & $9.0 \cdot {10}^{-4}$ \\ 
		40 & 12 & 3701 && $1.4 \cdot {10}^{-4}$ & $4.0 \cdot {10}^{-4}$ && $3.8 \cdot {10}^{-4}$ & $6.7 \cdot {10}^{-3}$ \\
		40 & 16 & 6613 && $5.7 \cdot {10}^{-5}$ & $3.8 \cdot {10}^{-4}$ && $2.9 \cdot {10}^{-3}$ & $2.8 \cdot {10}^{-1}$ \\ 
		40 & 20 & 10421 && $1.3 \cdot {10}^{-4}$ & $2.0 \cdot {10}^{-1}$ && $2.6 \cdot {10}^{-3}$ & $5.7 \cdot {10}^{-1}$ \\ \bottomrule
	\end{tabular}
\caption{Results of the detection step for important ANOVA terms of $f$ with $M = 10000$ Chebyshev distributed nodes $\X_{\text{cheb}}$ and uniformly distributed nodes $\X_{\text{uni}}$. We define $\tilde{f}_1 \coloneqq S(\X_{\text{cheb}}, \I(U_2))f$ and $\tilde{f}_2 \coloneqq S(\X_{\text{uni}}, \I(U_2))f$.}
\label{tab:res_det}
\end{table*}

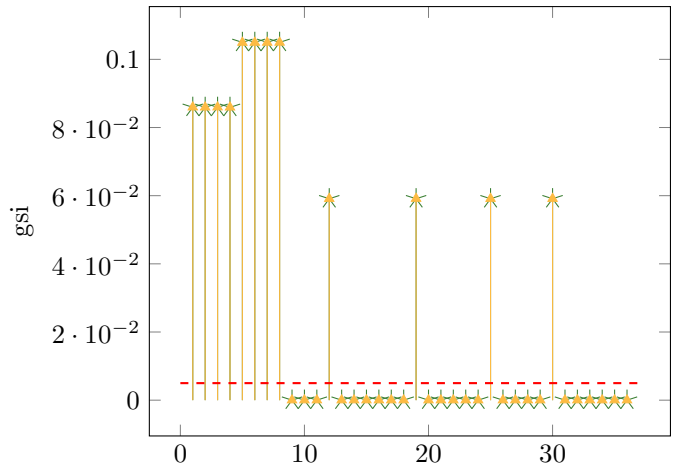
\begin{figure}[ht!]
	\begin{center}
		\centering
		\begin{tikzpicture}
			\begin{axis}[ylabel={gsi}]
				\addplot+[ycomb, OliveGreen, mark=star, mark options={OliveGreen, scale=2}] plot coordinates
				{( 1, 0.08592995960479803 )( 2, 0.08591394984599812 )( 3, 0.08592370784663406 )( 4, 0.08591524733018549 )( 5, 0.10497193710777422 )( 6, 0.10498180232329457 )( 7, 0.10498249061909459 )( 8, 0.10498410729081091 )( 9, 6.5232719510016385e-9 )( 10, 5.223413977217038e-9 )( 11, 4.596492921561453e-9 )( 12, 0.05909507440549067 )( 13, 6.624138837153116e-9 )( 14, 7.388679554760323e-9 )( 15, 5.933970622134937e-9 )( 16, 5.915657710745904e-9 )( 17, 5.209020291887556e-9 )( 18, 4.692234301658733e-9 )( 19, 0.0590994275083231 )( 20, 5.254066079586256e-9 )( 21, 7.973019935295283e-9 )( 22, 4.559764462524055e-9 )( 23, 4.811672305380702e-9 )( 24, 6.4188011503111255e-9 )( 25, 0.059097092894912454 )( 26, 7.744733334504471e-9 )( 27, 6.240101667578656e-9 )( 28, 6.367209370750284e-9 )( 29, 6.442484646316358e-9 )( 30, 0.0591050593628206 )( 31, 6.002083506619601e-9 )( 32, 5.863931575181954e-9 )( 33, 5.7489763015387054e-9 )( 34, 4.988728512292795e-9 )( 35, 7.126471679412714e-9 )( 36, 6.210938390859302e-9 )};
				\addplot+[ycomb, Dandelion, mark=triangle*, mark options={Dandelion}] plot coordinates
				{( 1, 0.08595107241538494 )( 2, 0.08592275680821813 )( 3, 0.08596324132549159 )( 4, 0.08591962373395347 )( 5, 0.10492415150952038 )( 6, 0.10492350641239008 )( 7, 0.1049397327114645 )( 8, 0.10495162652635859 )( 9, 1.8501417511032337e-8 )( 10, 2.2090186898335547e-8 )( 11, 1.9198441785364944e-8 )( 12, 0.05911832288820812 )( 13, 4.649892272574514e-8 )( 14, 3.8826616567496184e-8 )( 15, 4.734027922626823e-8 )( 16, 1.4488369408298909e-8 )( 17, 2.8493756943268905e-8 )( 18, 3.474217525300559e-8 )( 19, 0.05914176057653998 )( 20, 2.730180172505818e-8 )( 21, 2.5334762960753855e-8 )( 22, 1.9680822620612162e-8 )( 23, 1.439172237727401e-8 )( 24, 3.2882837771343375e-8 )( 25, 0.059122353741133324 )( 26, 2.4599350331131996e-8 )( 27, 1.394103125475747e-8 )( 28, 2.685369427219416e-8 )( 29, 2.91737387189762e-8 )( 30, 0.0591212018678416 )( 31, 3.978295364862823e-8 )( 32, 1.9204211832201168e-8 )( 33, 1.1263730856769696e-8 )( 34, 4.312436480994062e-8 )( 35, 2.1232460578358228e-8 )( 36, 3.0535845253613784e-8 )};
				\addplot[red,sharp plot,update limits=false, dashed, thick] 
				coordinates {(0,5e-3) (37,5e-3)};
			\end{axis}
		\end{tikzpicture}
	\end{center}
	\caption{Behavior of the global sensitivity indices $\gsi{\b u}{S(\X_{\text{cheb}}, \I(U_2))f}$ in green and $\gsi{\b u}{S(\X_{\text{uni}}, \I(U_2))f}$ in orange for parameters $N_1 = 20$, $N_2 = 8$.} 
	\label{fig:gsi}
\end{figure}

We observe that increasing the parameters $N_1, N_2$ will ultimately lead to an increasing generalization error $\eps_{\mathrm{L}_2}$. This is consistent with our results in \cref{sec:approx:lsqr} since we cannot guarantee with enough certainty that the system matrix in the normal equations has good properties if the index set size increases beyond a certain point. This effect appears sooner with the uniform nodes $\X_{\text{uni}}$ which is connected to the necessary preconditioning, see \cref{sec:uninodes}. In \cref{fig:gsi} we have visualized the global sensitivity indices for the parameters $N_1 = 20$ and $N_2 = 8$. The terms in $U^\ast$ are clearly separated from the terms in its complement, i.e., the active set detection worked well for this example. Moreover, we observe that the gsi obtained by approximation with uniform nodes and with Chebyshev nodes are very close together.

There clearly exist $N_1$, $N_2$, and $\b\varepsilon$, e.g.\, $N_1 = 20$, $N_2 = 8$, $\b\eps = (0.005,0.005)$, such that we are able to recover the set of ANOVA terms, i.e.,  $U_{\X_{\text{uni}},\b y}(\b\varepsilon) = U_{\X_{\text{cheb}},\b y}(\b\varepsilon) = U^\ast$. We aim to improve our approximation quality with the given data by solving the minimization problem with the active set $U^\ast$. Here, we could choose individual index sets for every ANOVA term in $U^\ast$ to form $\I(U^\ast)$ based on the global sensitivity indices, but for our function order-dependence can be maintained. \Cref{tab:res_approx} shows the results of the approximation using the index set $\I(U^\ast)$.

\begin{table*}[ht]\centering
	\begin{tabular}{@{}rrrrrrrrrrrrr@{}}\toprule \multicolumn{3}{c}{size of index sets} & \phantom{i}& \multicolumn{2}{c}{relative errors $\X_{\text{cheb}}$} & \phantom{i} &\multicolumn{2}{c}{relative errors $\X_{\text{uni}}$} \\\cmidrule{1-3} \cmidrule{5-6} \cmidrule{8-9}  $N_1$ & $N_2$ &  $\abs{I(U_{d_s})}$ && $\varepsilon_{\ell_2}^{\X_{\text{cheb}}}(f,\tilde{f}_1)$ & $\varepsilon_{\mathrm{L}_2}(f,\tilde{f}_1)$  && $\varepsilon_{\ell_2}^{\X_{\text{uni}}}(f,\tilde{f}_2)$ & $\varepsilon_{\mathrm{L}_2}(f,\tilde{f}_2)$  \\\midrule 
		60  & 12 & 957 && $1.6 \cdot {10}^{-4}$ & $3.8 \cdot {10}^{-4}$ && $1.6 \cdot {10}^{-4}$ & $4.1 \cdot {10}^{-4}$ \\ 
		60  & 20 & 1917 && $4.5 \cdot {10}^{-5}$ & $3.4 \cdot {10}^{-4}$ && $1.6 \cdot {10}^{-4}$ & $4.1 \cdot {10}^{-4}$ \\ 
		60  & 28 & 3389 && $1.8 \cdot {10}^{-5}$ & $3.4 \cdot {10}^{-4}$ && $6.9 \cdot {10}^{-4}$ & $6.9\cdot {10}^{-2}$ \\ \midrule
		80  & 12 & 1117 && $1.6 \cdot {10}^{-4}$ & $3.8 \cdot {10}^{-4}$ && $1.6 \cdot {10}^{-4}$ & $4.2 \cdot {10}^{-4}$ \\ 
		80  & 20 & 2077 && $4.5 \cdot {10}^{-5}$ & $3.4 \cdot {10}^{-4}$ && $7.1 \cdot {10}^{-4}$ & $7.2 \cdot {10}^{-2}$ \\ 
		80  & 28 & 3549 && $1.8 \cdot {10}^{-5}$ & $3.3 \cdot {10}^{-4}$ && $1.3 \cdot {10}^{-3}$ & $1.9 \cdot {10}^{-1}$ \\ \bottomrule
	\end{tabular}
	\caption{Results of approximation for important ANOVA terms of $f$ with $M = 10000$ Chebyshev distributed nodes $\X_{\text{cheb}}$ and uniformly distributed nodes $\X_{\text{uni}}$. We define $\tilde{f}_1 \coloneqq S(\X_{\text{cheb}}, \I(U^\ast))f$ and $\tilde{f}_2 \coloneqq S(\X_{\text{uni}}, \I(U^\ast))f$.}
	\label{tab:res_approx}
\end{table*}

We observe that the reduction of the ANOVA terms to $U^\ast$ yields a benefit with regard to the approximation quality. This results from the reduction in model complexity, i.e., we have larger oversampling factors for the same node set. The comparison of Chebyshev and uniform nodes yields a similar behavior as for the detection step. We achieve better errors for the Chebyshev distributed nodes without the addition of preconditioning which was to be expected since we choose a different sampling distribution for a weighted space.

\subsection{Friedman Benchmark Functions}\label{sec:friedman}

The Friedmann functions are a well-known example for the approximation of functions with scattered data, see e.g.~\cite{MeLeHo03, BeGaMo09, BiDaLa11}. We define the three non-periodic Friedmann functions as
\begin{align*}
	&\fun{\tilde{f}_1}{[0,1]^{10}}{\R}, && f_1(\x) = 10\sin(\pi x_1 x_2)+20(x_3-0.5)^2+10 x_4 + 5 x_5 \\
	&\fun{\tilde{f}_2}{[0,1]^{4}}{\R}, \, && f_2(\x) = \sqrt{ s_1^2(x_1) + \left( s_2(x_2) \cdot x_3 - \frac{1}{s_2(x_2) \cdot s_4(x_4)} \right)^2 } \\
	&\fun{\tilde{f}_3}{[0,1]^{4}}{\R}, \, && f_3(\x) = \arctan\left( \frac{s_2(x_2) \cdot x_3 - (s_2(x_2) \cdot s_4(x_4))^{-1} }{s_1(x_1)} \right)\\
\end{align*} with variable scalings $s_1(x_1) = 100x_1$, $s_2(x_2) = 520\pi x_2 + 40\pi$, and $s_4(x_4) = 10x_4+1$. The function $f_1$ has spatial dimension $10$. However, only five of the ten variables have any influence on the function. For $f_1$ and $f_2$ we also do not have more than two variables interact simultaneously, i.e., the superposition dimension for accuracy $\delta = 1$ is $2$, cf.\ \eqref{eq:anova:superposition}. For $f_3$ the superposition dimension for accuracy $\delta = 1$ would be equal to the spatial dimension $4$. Since the original functions are given on the interval $[0,1]^d$ we define for our experiments\begin{equation*}
	\fun{f_i}{[-1,1]^{d^{(i)}}}{\R}, \, f_i(\x) = \tilde{f}_i( 0.5(x_1+1), 0.5(x_2+1), \dots, 0.5(x_{d_i}+1) ),
\end{equation*} $i = 1, 2, 3$ with $d^{(1)} = 10$ and $d^{(2)} = d^{(3)} = 4$.

For calculation the approximations we use three sets of uniformly distributed nodes $\X_1 \subset (-1,1)^10$, $\X_2 \subset (-1,1)^4$, and $\X_3 \subset (-1,1)^4$ with $\abs{\X_i} = 200$. Moreover, we evaluate the functions at those nodes and additionally add Gaussian noise, i.e., \begin{equation*}
	\b{y}_i = \left( f(\x) + \eta_i \right)_{\x \in \X_i}, i = 1,2,3,
\end{equation*}  where the noise $\eta_i$ has zero mean and variances $\sigma_1 = 1$, $\sigma_2 = 125$, $\sigma_3 = 0.1$, respectively. In order to measure the error of an approximation $g_i$ on a Friedman function $f_i$, we consider the mean square error (mse) \begin{equation}
	\mathrm{mse}(f_i,g_i) = \frac{1}{1000} \sum_{\x \in \overline{\X}_i} \abs{f_i(\x)+\eta_i-g_i(\x)}^2
\end{equation} on additional sets of nodes $\overline{\X}_1 \subset (-1,1)^{10}$, $\overline{\X}_2 \subset (-1,1)^4$, and $\overline{\X}_3 \subset (-1,1)^4$ with $\abs{\overline{\X}_i} = 1000$.

\subsubsection{Friedman 1}

The first goal for Friedman 1 is to identify that the variables $x_6$ to $x_{10}$ have no contribution to the function. To this end we computed the approximation $S(\X_1, \I(U_2))f_1$ and considered the global sensitivity indices of the one-dimensional sets $\{i\}$, $i=1,2,\dots,10$. The result is depicted in \cref{fig:gsi_f1}. We observe that the variables $x_1$ to $x_5$ can be clearly separated from the rest. Therefore, we proceed with the active set \begin{equation}
	\tilde{U}_2 \coloneqq \{ \u \subset \{1,2,3,4,5\} \colon \au \leq 2 \}.
\end{equation}

Our second goal is to find the active set of terms $$U^\ast_1 = \{ \emptyset, \{1\}, \{2\} , \{3\} , \{4\} , \{5\} , \{1,2\}  \}.$$ To this end, we calculate the approximation $S(\X_1, \I(\tilde{U}_2))f_1$ with parameters $N_1 = N_2 = 4$. This yields a mean square error on the test nodes $\overline{\X}_1$ of $2.88$. As depicted in \cref{fig:gsi_f1:2} we are able to recover $U^\ast_1$ with a clear separation by choosing, e.g., $\b\eps = (0.03,0.03)$.

\begin{figure}
	\begin{subfigure}[t]{.45\textwidth}
		\centering
		\begin{tikzpicture}[scale=0.7]
			\begin{axis}[ylabel={gsi}]
				\addplot+[ycomb, OliveGreen, mark=star, mark options={OliveGreen, scale=2}] plot coordinates
				{(1, 0.15100375459344623) (2, 0.19405722054063987) (3, 0.09212246537139945) (4, 0.33432093891750747) (5, 0.1466966940116997) 				};
				\addplot+[ycomb, Dandelion, mark=triangle*, mark options={Dandelion}] plot coordinates
				{(6, 0.006379880635173035) (7, 0.0002751281447277956) (8, 0.0006918760010818275) (9, 0.0013039540011735935) (10, 0.004382263723088144)};
				\addplot[red,sharp plot,update limits=false, dashed, thick] 
				coordinates {(0,0.02) (11,0.02)};
			\end{axis}
		\end{tikzpicture}
		\caption{Global sensitivity indices $\gsi{\{i\}}{S(\X_1, \I(U_2))f_1}$, $i=1,2,3,4,5$, in green and $\gsi{\{i\}}{S(\X_1, \I(U_2))f_1}$, $i=6,7,8,9,10$, in orange for parameters $N_1 = 4$, $N_2 = 2$.}		\label{fig:gsi_f1}
	\end{subfigure}\hfill%
	\begin{subfigure}[t]{.45\textwidth}
		\centering
		\begin{tikzpicture}[scale=0.7]
			\begin{axis}[ylabel={gsi}]
				\addplot+[ycomb, OliveGreen, mark=star, mark options={OliveGreen, scale=2}] plot coordinates
				{(1, 0.15028855648551392) (2, 0.09077771791299585) (3, 0.10529892891061698) (4, 0.30192707364000215) (5, 0.14294077974105296) (6, 0.12349094405116816) };
				\addplot+[ycomb, Dandelion, mark=triangle*, mark options={Dandelion}] plot coordinates
				{(7, 0.004000444463240325) (8, 0.013358915462593966) (9, 0.012213644385319477) (10, 0.008755961428892251) (11, 0.0031610805927739473) (12, 0.010808010727823886) (13, 0.010759624068840252) (14, 0.01737792591285547) (15, 0.00484039221631036)};
				\addplot[red,sharp plot,update limits=false, dashed, thick] 
				coordinates {(0,0.03) (16,0.03)};
			\end{axis}
		\end{tikzpicture}
		\caption{Global sensitivity indices $\gsi{\u}{S(\X_1, \I(\tilde{U}_2))f_1}$, $\u \in U^\ast_1$ in green and $\u\in \tilde{U}_2\setminus U^\ast_1$ in yellow with parameters $N_1 = N_2 = 4$.}		\label{fig:gsi_f1:2}
	\end{subfigure}
\caption{Numerical experiments with the Friedman 1 function.}
\end{figure}
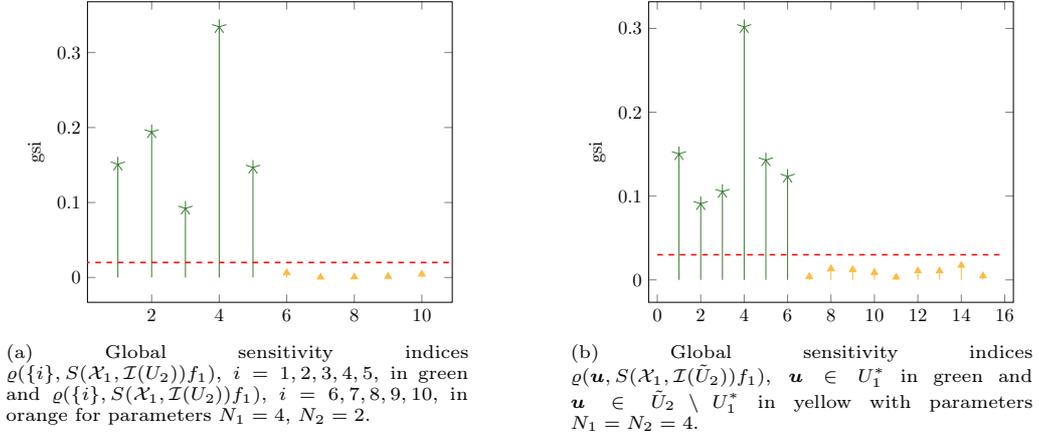

We computed the approximation $S(\X, \I(U_1^\ast))f_1$ with parameters $N_1 = N_2 = 4$ on 100 randomly generated uniformly i.i.d. pairs of node sets $(\X, \mathcal{T}) \subset (-1,1)^{10} \times (-1,1)^{10}$ with $\abs{\X} = 200$ and $\abs{\mathcal T} = 1000$. The mean square error was calculated on the sets $\mathcal T$. The median of the 100 mses is $1.17$. 

\subsubsection{Friedman 2}

For the Friedman 2 function we want to identify the active set of terms from our scattered data $\X_2$ and $\b{y}_2$. To this end, we computed the approximation $S(\X_2, \I(U_2))f_2$ with parameters $N_1 = N_2 = 2$. This yielded a mse on the data $\overline{\X}_2$ of $16.44 \cdot 10^3$. The resulting sensitivity indices are displayed in \cref{fig:gsi_f2}. We deduce that the terms \begin{equation}
	U^\ast_2 \coloneqq \{ \emptyset, \{2\} , \{3\} , \{2,3\}  \}
\end{equation} are clearly more important than the rest. They can be obtained by choosing a threshold vector, e.g., $\b\eps = (0.03,0.03)$.

\begin{figure}[ht!]
	\begin{center}
		\centering
		\begin{tikzpicture}[scale=0.75]
			\begin{axis}[ylabel={gsi}]
				\addplot+[ycomb, OliveGreen, mark=star, mark options={OliveGreen, scale=2}] plot coordinates
				{ (2, 0.3374202104336745) (3, 0.43513421746888004)  (8, 0.22057675207532537)  };
				\addplot+[ycomb, Dandelion, mark=triangle*, mark options={Dandelion}] plot coordinates
				{(1, 0.0031160244703633457) (9, 0.001342973528388098)(4, 0.00020614482773310283) (5, 0.0004122708547179238) (6, 0.0006754164855170884) (7, 0.0008477714622283992) (10, 0.0002682183931721676)};
				\addplot[red,sharp plot,update limits=false, dashed, thick] 
				coordinates {(0,0.03) (16,0.03)};
			\end{axis}
		\end{tikzpicture}
	\end{center}
	\caption{Global sensitivity indices $\gsi{\u}{S(\X_2, \I(U_2))f_2}$, $\u \in U^\ast_2$ in green and $\u\in U_2\setminus U^\ast_2$ in yellow with parameters $N_1 = N_2 = 2$.} 
	\label{fig:gsi_f2}
\end{figure}
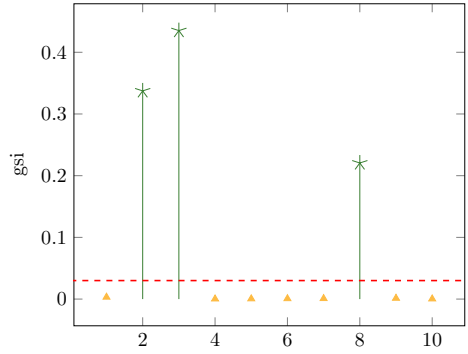

We computed the approximation $S(\X, \I(U_2^\ast))f_2$ with parameters $N_1 = N_2 = 2$ on 100 randomly generated uniform i.i.d.\ pairs of node sets $(\X, \mathcal{T}) \subset (-1,1)^{10} \times (-1,1)^{10}$ with $\abs{\X} = 200$ and $\abs{\mathcal T} = 1000$. The mean square error was calculated on the sets $\mathcal T$. The median of the 100 mses is $16.09 \cdot 10^3$. 

\subsubsection{Friedman 3}

As before, we first aim to identify an active set of terms from the scattered data $\X_3$ and $\b{y}_3$. The approximation $S(\X_3, \I(U_2))f_3$ with parameters $N_1 = 8$,  $N_2 = 2$ yielded a mean square error of $1.8 \cdot 10^{-2}$ on the node set $\overline{X}_3$. The sensitivity indices are displayed in \cref{fig:gsi_f3}. From this we identify the active set as \begin{equation}
	U^\ast_3 \coloneqq \{ \emptyset, \{1\} , \{2\} ,  \{3\},  \{1,2\} , \{1,3\} , \{2,3\}  \},
\end{equation} e.g., with a choice of $\b\eps = (0.002,0.002)$.

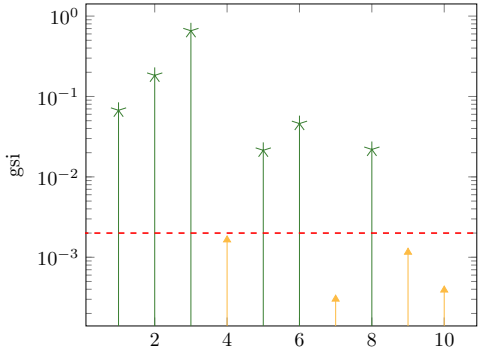
\begin{figure}[ht!]
	\begin{center}
		\centering
		\begin{tikzpicture}[scale=0.75]
			\begin{axis}[ylabel={gsi},ymode=log, log origin=infty]
				\addplot+[ycomb, OliveGreen, mark=star, mark options={OliveGreen, scale=2}] plot coordinates
				{(1, 0.06733958609733687) (2, 0.1831868334697192) (3, 0.6566338208242045)  (5, 0.021392184389478423) (6, 0.0460206245214079)  (8, 0.02193296692690956)};
				\addplot+[ycomb, Dandelion, mark=triangle*, mark options={Dandelion}] plot coordinates
				{ (4, 0.0016456207721343152) (7, 0.00030161904821638383) (9, 0.0011555820641467587) (10, 0.00039116188644626974) };
				\addplot[red,sharp plot,update limits=false, dashed, thick] 
				coordinates {(0,0.002) (16,0.002)};
			\end{axis}
		\end{tikzpicture}
	\end{center}
	\caption{Global sensitivity indices $\gsi{\u}{S(\X_3, \I(U_2))f_3}$, $\u \in U^\ast_3$ in green and $\u\in U_2\setminus U^\ast_3$ in yellow with parameters $N_1 =8$, $N_2 = 2$.} 
	\label{fig:gsi_f3}
\end{figure}

We performed the approximation $S(\X, \I(U_3^\ast))f_3$ with parameters $N_1 =8$, $N_2 = 2$ on 100 randomly generated uniform i.i.d. pairs of node sets $(\X, \mathcal{T}) \subset (-1,1)^{4} \times (-1,1)^{4}$ with $\abs{\X} = 200$ and $\abs{\mathcal T} = 1000$. The median of the mean square error on the sets $\mathcal T$ is $17.22 \cdot 10^{-3}$. 

\subsubsection{Comparison}

\cref{tab:friedman} contains the benchmark data from \cite{MeLeHo03} with a support vector machine (SVM), a linear model (lm), a neural network (mnet) and a random forest (rForst) as well as the results with our method (ANOVAapprox). We are able to achieve a more accurate approximation in the exact same setting for every one of the three functions. The value for ANOVAapprox was obtained by computing the model on 100 randomly generated node sets and computing the error on 100 randomly generated test sets. 

\begin{table}[tbhp]
	\begin{center}
		\begin{tabular}{cccccc} 
			\toprule
			& svm & lm & mnet & rForst & ANOVAapprox \\ 
			\midrule 
			Friedman 1 & 4.36 & 7.71 & 9.21 & 6.02 & \textbf{1.17} \\
			Friedman 2 ($\cdot \, 10^3$) & 18.13 & 36.15  & 19.61  & 21.50 &  \textbf{16.09}  \\
			Friedman 3 ($\cdot \, 10^{-3}$) & 23.15 & 45.42 & 18.12 & 22.21 & \textbf{17.22} \\
			\bottomrule
		\end{tabular}
		\caption{Mean squared errors (MSE) for different methods when approximating Friedman functions in \cite{MeLeHo03} compared to  our method (ANOVAapprox). All values are the medians of the experiment MSEs and the best value for every function is highlighted.}\label{tab:friedman}
	\end{center}
\end{table} 

\section{Summary}

In this paper we considered the classical ANOVA decomposition for functions $f$ in weighted Lebesgue spaces $\Lomd$ with orthogonal polynomials as bases. Specifically, we proved relations between the basis coefficients of the projections $\mathrm{P}_{\u} f$, the ANOVA terms $f_{\u}$, and the function $f$. Furthermore, we considered sensitivity analysis and truncating the ANOVA decomposition to a certain subset of terms.

We introduced a method to determine important ANOVA terms, i.e., terms with a high global sensitivity index $\gsi{\u}{f}$, by approximation with index sets with a low-dimensional structure related to the truncated ANOVA decomposition. Our scenario was scattered data approximation where only a node set $\X$ and possibly noisy function values $\b y = (f(\x))_{\x\in \X}$ are known. Properties of the corresponding normal equations were considered om the case of nodes distributed according to the Chebyshev density. We also introduced preconditioning for uniformly distributed nodes and considered properties of the resulting system as well. 

The numerical experiments show that the method works using a specific test function consisting of sums of products of B-splines. The test function had a superposition dimension of $3$ for an arbitrary accuracy, i.e., $\mathrm{T}_3 f = f$, and we were able to recover the active set of ANOVA terms with our approach. Experiments with the Friedman functions showed that we proposed a competitive method that yields better results on these functions as other well-known methods such as support vector machines.

\section*{Acknowledgments}
We thank Tino Ullrich and Toni Volkmer for fruitful discussions on the contents of this paper. Daniel Potts acknowledges funding by Deutsche Forschungsgemeinschaft (German Research Foundation) -- Project--ID 416228727 -- SFB 1410. Michael Schmischke is supported by the BMBF grant 01$|$S20053A. 

\bibliographystyle{elsarticle-num}
\bibliography{references}

\end{document}